\documentclass[final,1p]{elsarticle}

\usepackage{amssymb}
\usepackage{amsmath}
\usepackage{amsthm}
\usepackage{lineno}

\usepackage{latexsym}
\usepackage{bm}
\usepackage[ruled,lined]{algorithm2e}
\usepackage{algorithmic}
\usepackage{array}
\usepackage[colorlinks,linkcolor=blue, anchorcolor=blue, citecolor=blue]{hyperref}
\usepackage{color,tabularx}
\usepackage{booktabs}
\usepackage{subfigure}
\usepackage{multirow}
\usepackage{stmaryrd}
\usepackage[T1]{fontenc}
\biboptions{numbers,sort&compress}
\usepackage{graphicx} 
\usepackage{adjustbox}

\newcommand{\f}{\frac}

\newcommand{\mO}{\mathcal{O}}
\newcommand{\mL}{\mathcal{L}}

\newcommand{\mS}{\mathcal{S}}
\numberwithin{equation}{section}
\allowdisplaybreaks
\newtheorem{theorem}{Theorem}[section]
\newtheorem{lemma}[theorem]{Lemma}

\theoremstyle{definition}

\theoremstyle{remark}
\newtheorem{remark}{Remark}
\journal{XXXX}

\begin{document}
	
    \begin{frontmatter}
		
    \title{Numerical analysis and efficient implementation of fast collocation methods for fractional Laplacian model on nonuniform grids}
        \author[1]{Meijie Kong}\ead{kmj@stu.ouc.edu.cn}
		\author[1,2]{Hongfei Fu\corref{Fu}}\ead{fhf@ouc.edu.cn}
		\address[1]{School of Mathematical Sciences, Ocean University of China, Qingdao, Shandong 266100, China}
		\address[2]{Laboratory of Marine Mathematics, Ocean University of China, Qingdao, Shandong 266100, China}
		\cortext[Fu]{Corresponding author.}

\begin{abstract}
 We propose a fast collocation method based on Krylov subspace iterative solver on general nonuniform grids for the fractional Laplacian problem, in which the fractional operator is presented in a singular integral formulation. The method is proved to be uniquely solvable on general nonuniform grids for $\alpha\in(0,1)$, provided that the sum-of-exponentials (SOE) approximation is sufficiently accurate. In addition, a modified scheme is developed and proved to be uniquely solvable on uniform grids for $\alpha\in(0,2)$. Efficient implementation of the proposed fast collocation schemes based on fast matrix-vector multiplication is carefully discussed, in terms of computational complexity and memory requirement. To further improve computational efficiency, a banded preconditioner is incorporated into the Krylov subspace iterative solver. A rigorous maximum-norm error analysis for $\alpha\in(0,1)$ is presented on specific graded grids, which shows that the convergence order depends on the grading parameter. Numerical experiments validate the predicted convergence and demonstrate the efficiency of the fast collocation schemes.
\end{abstract}	

\end{frontmatter}

\section{Introduction}
The fractional Laplacian serves as a prototypical operator for modeling nonlocal diffusion phenomena. Such nonlocal diffusion equations, which incorporate long range interactions, have attracted considerable interest in recent years \cite{DV21,DG12,VA96}. Specifically, Dipierro and Valdinoci \cite{DP23} described a model in which a nonlocal operator is coupled with a local operator, capturing the behavior of foraging animals that alternate between long-jump movements and local searching patterns. Moreover, the fractional Laplacian operator has been widely applied in various mathematical models, such as the fractional Allen–Cahn, Cahn–Hilliard, and porous medium equations \cite{AS16, AM17}. It also serves as a powerful tool for modeling complex physical phenomena, such as anomalous diffusion \cite{Y12}, quasi-geostrophic flows \cite{BN21}, and drug transport \cite{PM11}. 
Meanwhile, there has also been some great progress in numerically solving the related nonlocal models, including finite difference methods \cite{DZ19,HO14,HZ21}, finite element methods \cite{AB17,BL19}, collocation methods \cite{CQ21,ZG16,AD21}, mesh-free methods \cite{BW21,RD19} and so on.

 This paper is concerned with efficient numerical treatments of the following one-dimensional nonlocal model problem
\begin{equation}\label{model:nonlocal}
\left\{
\begin{aligned}
	(-\Delta)^{\frac{\alpha}{2}}u & =f(x),  &&\quad \text{in}\ (a,b),\\
	u & = 0,                          &&\quad \text{in}\ (-\infty, a] \cup [b, +\infty). 
\end{aligned}\right.
\end{equation}
where the fractional Laplacian $(-\Delta)^{\frac{\alpha}{2}}$ with $\alpha\in(0,2)$ can be defined in several equivalent ways on the whole space $\mathbb{R}$ \cite{K17}, and here we recall its hypersingular integral form
\begin{equation}\label{model:Lap_Def}
(-\Delta)^{\frac{\alpha}{2}}u(x)=C_{\alpha}\,\textit{P.V.}\int_{\mathbb{R}}\frac{u(x)-u(y)}{|x-y|^{1+\alpha}}dy,
\end{equation}
where \textit{P.V.} stands for the Cauchy principal value and 
$
C_{\alpha}=\frac{\alpha2^{\alpha-1} \Gamma(\frac{1+\alpha}{2})}{\pi^{1/2}\Gamma(1-\frac{\alpha}{2})}
$
is a normalization constant.

Note that the definition in \eqref{model:Lap_Def} reveals the nonlocal nature of the fractional Laplacian operator, and in order to evaluate the operator for a given function $u$ at position $x$, a singular integral over the whole space $\mathbb{R}$ has to be accurately evaluated. However, it is well-known that any numerical discretization of the fractional Laplacian shall result in a dense stiffness matrix, which in turn causes huge computational cost in large-scale modeling and simulations. Notably, several studies have focused on fast solution methods for such nonlocal model \eqref{model:nonlocal} discretized on \textit{uniform} spatial grids. For instance, Duo and Zhang \cite{DZ19} reformulated the fractional operator as a weighted integral of central difference quotients, and proposed a fractional central difference scheme for the two and three dimensional model problems \eqref{model:nonlocal}. Their discretization produces a symmetric multilevel Toeplitz stiffness matrix. Minden and Ying \cite{MY20} also introduced a simple method to discretize the integral operator, which also leads to a Toeplitz matrix. Along with a preconditioner based on the finite-difference Laplacian, the underlying linear system is efficiently solved via preconditioned Krylov methods. However, the convergence analysis relies on strong regularity assumptions on the exact solution, which may not hold in general.
 Antil et al. \cite{AD21} avoided direct discretization of the integral kernel; instead, they employed a Fourier–sinc approximation to represent the fractional Laplacian operator. All these mentioned works permit the development of fast Fourier transform (FFT)-based fast algorithms, which reduces the computational cost to $\mO(N\log N)$ per iteration, and meanwhile, the memory requirement is reduced to $\mO(N)$, where $N$ is the total number of unknowns.
 
Another significant challenge in designing numerical methods for nonlocal problems is the underlying low solution regularity. Though the convergence analysis of most existing numerical schemes rely on smooth assumptions for the exact solutions; nevertheless, in practice, the solutions often exhibit weak boundary singularities. Such singular behaviors naturally arise in problems with singular integral kernels, such as the time-fractional equations \cite{SY11}, the space-fractional equations \cite{EHR18,WY17}, and the nonlocal/regional fractional Laplacian \cite{AB18,F22,HW22,RS14}. According to Corollary 1 in \cite{CD24}:
if $f\in C^{\beta}(\Omega)$, where $\beta=2-\alpha+\gamma$ with $\gamma<\alpha<1$, then, there exists a constant $C = C(\Omega,\alpha,\beta,f)$ such that 
\begin{equation}\label{reg:ass}
	\Big|\frac{\partial^{\ell}}{\partial x^{\ell}} u(x)\Big|\lesssim[(x-a)(b-x)]^{\frac{\alpha}{2}-\ell},  ~\ell=0, 1, 2
\end{equation}
holds for model \eqref{model:nonlocal}. Therefore, to resolve singularities like that in \eqref{reg:ass} and enhance numerical accuracy, graded meshes are commonly employed  \cite{HW22,RS14,SO17}. Nevertheless, for  nonuniform spatial grids, the resulting coefficient matrices of corresponding numerical methods lose their Toeplitz-like structures, and thus, the earlier mentioned FFT-based fast algorithms are no longer applicable. While, traditional direct solvers such as Gaussian elimination method for any numerical discretizations shall require $\mO(N^2)$ memory and $\mO(N^3)$ computational complexity. Even if a Krylov subspace iterative method is adopted, it still costs $\mO(N^2)$ computational complexity per iteration. Therefore, it is yet challenging to construct efficient algorithms for the nonlocal model \eqref{model:nonlocal} on arbitrary nonuniform grids.

Motivated by these challenges in the modeling of fractional Laplacian equation \eqref{model:nonlocal}, this paper focuses on developing fast collocation methods on general nonuniform grids and presenting a rigorous error analysis under low regularity assumptions. The methods are constructed based on the so-called sum-of-exponentials (SOE) approximation approach, see Lemma \ref{lem:soe} below, which enables an efficient representation of the nonlocal kernel on nonuniform grids. By carefully exploiting the SOE structure, fast matrix-vector multiplications are achieved for the resulting dense coefficient matrices. As a result, the Krylov subspace iterative solvers can be efficiently implemented without requirement of any matrix assembling, allowing numerical solutions to be obtained in only $\mO(N\log^2 N)$ operations per iteration. The proposed fast collocation scheme is rigorously proven to be uniquely solvable for $\alpha\in(0,1)$ on nonuniform grids, if a sufficiently accurate SOE approximation error is chosen. However, for $\alpha\in(1,2)$, the unique solvability of the original fast scheme cannot be theoretically guaranteed, even on uniform grids. To address this issue, a modified fast collocation scheme is further introduced by modifying the local integral approximation, and it is proven to be uniquely solvable on uniform grids for the extended range $\alpha\in(0,2)$. Furthermore, a banded preconditioner is introduced to reduce the number of iterations and accelerate the iterative process. Most importantly, a rigorous maximum-norm error analysis is carried out for both the original and modified fast collocation schemes on symmetric graded grids with $\alpha \in (0,1)$, demonstrating that the convergence order depends on the grading parameter. Numerical experiments are presented to test the performance of the two proposed fast schemes. It shows that for $\alpha\in(0,1)$, the modified fast scheme yields smaller errors on uniform grids, while the original scheme performs better on nonuniform grids--particularly for small values of $\alpha$; for $\alpha \in [1,2)$, the modified scheme attains the expected optimal-order convergence, providing significantly more accurate and stable results compared to the less satisfactory approximations from the original scheme.

The rest of the paper is organized as follows. In Section \ref{sec:scheme}, we present a fast collocation scheme for model \eqref{model:nonlocal} discretized on general nonuniform grids, and also analyze its unique solvability for $\alpha\in(0,1)$. A modified fast scheme is then developed, and its unique solvability on uniform grids is discussed. Finally, efficient implementation of the proposed fast collocation schemes based on fast matrix-vector multiplication is carefully discussed, in terms of computational complexity and memory requirement. In Section \ref{sec:ana}, we provide a detailed analysis of the truncation error of the proposed fast methods, which leads to the main error estimate under lower regularity assumptions.  Section \ref{sec:num} is dedicated to validating the theoretical results through a series of numerical experiments. Finally, we draw concluding remarks in the last section.

\section{Fast collocation methods and efficient implementations}\label{sec:scheme}
In this section,  an SOE-based fast collocation scheme, depending on the fractional order $\alpha$, is first proposed for the model problem \eqref{model:nonlocal}, and then its unique solvability on general nonuniform grids is analyzed for case $\alpha \in (0,1)$.
Furthermore, a modified fast collocation scheme is proposed by developing a new local approximation, and unique solvability on uniform grids is analyzed for case $\alpha\in [1,2)$. Finally, efficient implementations of these collocation methods are developed within the framework of Krylov subspace iterative solvers, leveraging fast matrix-vector multiplication and a preconditioning technique.

\subsection{Numerical scheme}
Let $N$ be a positive integer. We introduce a general nonuniform spatial partition  
\begin{equation*}
  a=x_0<x_1<x_2<\cdots<x_{N-1}<x_N=b
\end{equation*}
with step sizes $h_j=x_j-x_{j-1}$ for $j=1,2,\ldots,N$ and $h:=\max_jh_j$. 
Let $ \mS_h(a,b)$ denote the space of piecewise linear continuous functions with respect to the given partition, which vanishes at the boundaries $x=a$ and $x=b$. Then, the piecewise linear interpolation of the true solution $u(x)$ is defined as
\begin{equation*}
\Pi_h u(x):=\sum_{j=1}^{N-1}u_j\phi_j(x) \in  \mS_h(a,b),
\end{equation*}
where $u_j:=u(x_j)$ and $\{\phi_j(x)\}_{j=1}^{N-1}$ are the linear nodal basis functions such that
\begin{equation*}
\phi_j(x)=\left\{ 
    \begin{aligned}
	    &\frac{x-x_{j-1}}{x_j-x_{j-1}},  &x \in[x_{j-1}, x_j], \\
	    &\frac{x_{j+1}-x}{x_{j+1}-x_j},  &x \in[x_j, x_{j+1}], \\ 
	    &0,  &\text {otherwise.}
	\end{aligned}\right.
\end{equation*}

Replacing $u(x)$ by $\Pi_hu(x)$ in the nonlocal operator $(-\Delta)^{\frac{\alpha}{2}} u$, and taking values at the collocation points $\{x_i\}_{i=1}^{N-1}$, we obtain
\begin{equation}\label{Lap_approx:e1}
(-\Delta)^{\frac{\alpha}{2}} u(x_i) \approx C_{\alpha}\int_{\mathbb{R}}\frac{u(x_i)-\Pi_hu(y)}{|x_i-y|^{1+\alpha}}dy
=:\mL [u]_i.
\end{equation}
To develop an efficient collocation method for model \eqref{model:nonlocal}, we decompose the nonlocal integration in \eqref{Lap_approx:e1} into three distinct parts as follows
\begin{equation}\label{Lap_approx:e2}
   \begin{aligned}
	\mL [u]_i& = C_{\alpha} \bigg[\int_{-\infty}^{x_{i-1}} \frac{ u(x_i)-\Pi_hu(y)}{(x_i-y)^{1+\alpha}}dy
                     +\int_{x_{i-1}}^{x_{i+1}} \frac{ u(x_i)-\Pi_hu(y)}{|x_i-y|^{1+\alpha}}dy \\
                     & \qquad \qquad + \int_{x_{i+1}}^{+\infty} \frac{ u(x_i)-\Pi_hu(y)}{(y-x_i)^{1+\alpha}}dy\bigg] \\
		&=: C_{\alpha}\Big(\mL_{l}[u]_i+\mL_{loc}[u]_i+\mL_{r}[u]_i\Big),
	\end{aligned}
\end{equation}
where $\mL_{l}[u]$ and $\mL_{r}[u]$ correspond to the left-sided and right-sided nonlocal integration, respectively, and $\mL_{loc}[u]$ represents integration over the local interval $[x_{i-1}, x_{i+1}]$.

Direct and exact computations of the integrals in \eqref{Lap_approx:e2} can lead to a collocation scheme for model \eqref{model:nonlocal}, however, the scheme is computationally inefficient due to the nonlocal parts $\mL_{l}[u]$ and $\mL_{r}[u]$.
To reduce computational complexity and memory requirement, 
we employ the following sum-of-exponentials (SOE) approximation technique to replace the weakly singular kernel function $x^{-1-\alpha}$ in the nonlocal integrals. 
\begin{lemma}[\cite{JZZZ17}]\label{lem:soe}
For given $\beta \in(0,2)$, an absolute tolerance error $\epsilon$, a cut-off restriction $\Delta x>0$ and a given position $X>0$, there exists a positive integer $N_{e}$, positive quadrature points $\{\lambda_s\}_{s=1}^{N_{e}}$ and corresponding positive weights $\{\theta_s\}_{s=1}^{N_{e}}$ satisfying
\begin{equation}\label{equ:soe}
    \Big|x^{-\beta}-\sum_{s=1}^{N_{e}}\theta_s e^{-\lambda_s x}\Big|\le \epsilon,\quad\forall x\in[\Delta x,X],
\end{equation}
where the number of exponentials satisfies
\begin{equation*}
N_{e}=\mO\left(\log\frac{1}{\epsilon}\Big(\log\log\frac{1}{\epsilon}+\log\frac{X}{\Delta x}\Big)+\log\frac{1}{\epsilon}\Big(\log\log\frac{1}{\epsilon}+\log\frac{1}{\Delta x}\Big)\right).
\end{equation*}
\end{lemma}

In the following, we first pay attention to the fast approximations of the nonlocal integration terms in \eqref{Lap_approx:e2} mainly using Lemma \ref{lem:soe}. Without loss of generality, we focus below on deriving the fast approximation formula for the left-sided nonlocal integral $\mL_{l}[u]_i$ using the SOE technique, while it is analogous for the nonlocal integral $\mL_{r}[u]_i$. The discussion is divided into two different cases regarding the fractional order $\alpha$. 

Firstly, for the case $\alpha\in(0,1)$, we replace the convolution kernel $(x_i-y)^{-1-\alpha}$ by its SOE approximation in \eqref{equ:soe}, then the left-sided nonlocal integral is approximated as follows
\begin{equation}\label{Lap_approx:e3}
\begin{aligned}
    \mL_{l}[u]_i & \approx\int_{-\infty}^{x_{i-1}}\frac{u(x_i)}{(x_i-y)^{1+\alpha}}dy - \int_{-\infty}^{x_{i-1}}\sum_{s=1}^{N_{e}}\theta_se^{-\lambda_s(x_i-y)} \Pi_hu(y)dy \\
    & =\frac{h_i^{-\alpha}}{\alpha} u_i - \sum_{s=1}^{N_{e}}\theta_s \mS_{i,s}^{L}[u], \quad \mS_{i,s}^{L}[u] := \int_{-\infty}^{x_{i-1}}e^{-\lambda_s(x_i-y)} \Pi_hu(y) dy,
    \end{aligned}
\end{equation}
where the parameters $\{\lambda_s, \theta_s, N_{e}\}$ correspond to the case $\beta= 1+\alpha$ in \eqref{equ:soe}.
To compute the integral $\mS_{i,s}^{L}[u]$ efficiently, we shall develop a recursive formula by splitting the integral again into nonlocal and local parts and then treating them in different approaches, i.e.,
\begin{equation}\label{Lap_approx:e3a}
\begin{aligned}
	\mS_{i,s}^{L} [u]& = \int_{-\infty}^{x_{i-2}}e^{-\lambda_s(x_i-y)}\Pi_hu(y)dy + \int_{x_{i-2}}^{x_{i-1}}e^{-\lambda_s(x_i-y)}\Pi_hu(y)dy\\
         & = e^{-\lambda_s h_i}\int_{-\infty}^{x_{i-2}}e^{-\lambda_s(x_{i-1}-y)}\Pi_hu(y)dy 
         + \sum_{j=1}^{N-1}u_j\int_{x_{i-2}}^{x_{i-1}}e^{-\lambda_s(x_i-y)}\phi_j(y) dy\\
		& = \omega_{i,s} \mS_{i-1,s}^{L}[u] + \mu_{i,s}^{L}u_{i-1} + \nu_{i,s}^{L}u_{i-2}, \quad 2\leq i\leq N-1,
\end{aligned}
\end{equation}
where, by definition and the boundary condition, $\mS_{1,s}^{L}[u]  = 0$, and the coefficients are defined as
\begin{equation*}
	\left\{
	\begin{aligned}
		\omega_{i,s} & = e^{-\lambda_sh_i},\\
		\mu_{i,s}^{L} & 
          = \f{\omega_{i,s}(\lambda_s h_{i-1} +\omega_{i-1,s}-1)}{\lambda_s^2h_{i-1}},\\
		\nu_{i,s}^{L} &
         =-\f{\omega_{i,s}(\lambda_s h_{i-1}\omega_{i-1,s} +\omega_{i-1,s}-1)}{\lambda_s^2h_{i-1}}.
	\end{aligned}\right.
\end{equation*}

Secondly, we discuss the evaluation of the nonlocal integral for the case $\alpha \in [1,2)$. Since in this case the singular kernel index in \eqref{Lap_approx:e2} is now $1+\alpha\in[2,3)$, Lemma \ref{lem:soe} cannot be applied directly. Instead, we first evaluate it using integration by parts that 
\begin{equation*}
  \mL_l[u]_i
     =\frac{h_i^{-\alpha}}{\alpha}(u_i-u_{i-1})
      + \f{1}{\alpha} \int_{-\infty}^{x_{i-1}} \frac{\big[\Pi_hu(y)\big]'}{(x_i-y)^\alpha}dy.
\end{equation*}
Then, by approximating the above convolution kernel $(x_i-y)^{-\alpha}$ with the SOE formula, it yields the numerical approximation 
\begin{equation}\label{Lap_approx:e4}
\begin{aligned}
\mL_{l}[u]_i \approx \frac{h_i^{-\alpha}}{\alpha}(u_i-u_{i-1}) 
     + \frac{1}{\alpha}\sum_{s=1}^{\tilde{N}_{e}}\tilde{\theta}_s\tilde{ \mS}_{i,s}^{L}[u],\quad 
\tilde{\mS}_{i,s}^L[u]:= \int_{-\infty}^{x_{i-1}}e^{-\tilde{\lambda}_s(x_i-y)} \big[\Pi_hu(y)\big]' dy, 
\end{aligned}
\end{equation}
where the parameters $\{\tilde{\lambda}_s, \tilde{\theta}_s, \tilde{N}_{e}\}$ correspond to the case $\beta= \alpha$ in \eqref{equ:soe}, and similarly, the integral $\tilde{\mS}_{i,s}^L[u]$ can also be computed recursively as
\begin{equation}\label{Lap_approx:e4a}
  \begin{aligned}
         \tilde{ \mS}_{1,s}^{L}[u]  = 0,~~
		 \tilde{ \mS}_{i,s}^{L}[u] 
		 = \tilde{\omega}_{i,s}\tilde{ \mS}_{i-1,s}^{L}[u] + \tilde{\mu}_{i,s}^{L} (u_{i-1} - u_{i-2}), \quad 2\leq i\leq N-1,
\end{aligned}
\end{equation}
with
\begin{equation*}
\begin{aligned}
\tilde{\omega}_{i,s}  = e^{-\tilde{\lambda}_sh_i},\quad
\tilde{\mu}_{i,s}^{L}  = \frac{\tilde{\omega}_{i,s}(1-\tilde{\omega}_{i-1,s})}{\tilde{\lambda}_sh_{i-1}}.
\end{aligned}
\end{equation*}

Similarly, by using the same approach as above, the fast approximation for the right-sided nonlocal integral $\mL_{r}[u]_i$ is proposed as follows
\begin{equation}\label{Lap_approx:e5}
    \mL_{r}[u]_i \approx \left\{\begin{aligned}
    &\frac{h_{i+1}^{-\alpha}}{\alpha} u_i - \sum_{s=1}^{N_{e}}\theta_s \mS_{i,s}^{R}[u], &&\quad\alpha\in(0,1),\\
     & \frac{h_{i+1}^{-\alpha}}{\alpha}(u_i-u_{i+1}) 
     + \frac{1}{\alpha}\sum_{s=1}^{\tilde{N}_{e}}\tilde{\theta}_s\tilde{ \mS}_{i,s}^{R}[u], &&\quad\alpha\in[1,2),
    \end{aligned}\right.
\end{equation}
where $\mS_{N-1,s}^{R}[u]=\tilde{ \mS}_{N-1,s}^{R}[u]=0$, and the integrals $\{ \mS_{i,s}^{R}\}$ and  $\{\tilde{\mS}_{i,s}^{R}\}$ can be computed recursively for $i= N-2, \ldots, 1$ as
\begin{equation}\label{Lap_approx:e5a}
   \left\{\begin{aligned}
     \mS_{i,s}^{R}[u] &:= \int^{+\infty}_{x_{i+1}}e^{-\lambda_s(y-x_i)} \Pi_hu(y) dy
        =\omega_{i+1,s} \mS_{i+1,s}^{R}[u] + \mu_{i,s}^{R}u_{i+1} + \nu_{i,s}^{R}u_{i+2}, \\
     \tilde{\mS}_{i,s}^{R}[u] &:= -\int^{+\infty}_{x_{i+1}}e^{-\tilde{\lambda}_s(y-x_i)}  \big[\Pi_hu(y)\big]'  dy
      =\tilde{\omega}_{i+1,s} \tilde{\mS}_{i+1,s}^{R}[u] + \tilde{\mu}_{i,s}^{R} (u_{i+1} - u_{i+2}),
   \end{aligned}\right.
\end{equation}
with coefficients defined by
\begin{equation*}
\left\{
	\begin{aligned}
		\mu_{i,s}^{R} & 
        =\frac{\omega_{i+1,s}(\lambda_sh_{i+2}+\omega_{i+2,s}-1)}{\lambda_s^2h_{i+2}},\\
		\nu_{i,s}^{R} & 
        =-\frac{\omega_{i+1,s}(\lambda_sh_{i+2}\omega_{i+2,s}+\omega_{i+2,s}-1)}{\lambda_s^2h_{i+2}},\\
            \tilde{\mu}_{i,s}^{R} &= \frac{\tilde{\omega}_{i+1,s}(1-\tilde{\omega}_{i+2,s})}{\tilde{\lambda}_sh_{i+2}}. 
	\end{aligned}\right.
\end{equation*}

Next, we discuss the approximation of the local integration term in \eqref{Lap_approx:e2}. Since the local part $\mL_{loc}[u]_i$ contributes negligible memory requirement and computational cost compared to the nonlocal parts $\mL_{l}[u]_i$ and $\mL_{r}[u]_i$, it can be computed directly. However, classical quadrature rules cannot be directly applied due to the strong singularity of the kernel function involved. Here we introduce a linear operator $\mL^*[g](x_i)$ \cite{VL02,ZG16}, which serves as a surrogate operator for the original local integral $\mL_{loc}[u] (x_i)$, defined as
 \begin{equation}\label{Local_approx:e1}
      \mL_{loc}[u]_i= \mL^*[g](x_i) 
                   = \lim_{\delta\rightarrow 0}\Big(\int_{\Omega_\delta}\frac{g(y)}{|x_i-y|^{1+\alpha}}dy
                      +r(\delta)\Big), \quad  g(y)= u(x_i)-\Pi_hu(y),
 \end{equation}
 where $\Omega_\delta=(x_{i-1},x_{i+1})\setminus(x_{i}-\delta,x_i+\delta)$ and
 \begin{equation}\label{Local_approx:e2}
   r(\delta)
     =\left\{\begin{aligned}
        & \frac{\delta^{-\alpha}}{-\alpha} \left(g(x_{i}^{-})+g(x_{i}^{+})\right)=0, && \quad \alpha\in(0,1),\\
       & -\delta^{-1} \left(g(x_{i}^{-})+g(x_{i}^{+})\right)-\ln\delta \left(g'(x_{i}^{-})-g'(x_{i}^{+}) \right)\\
       &\qquad=-\ln\delta \left(\frac{u_{i-1}}{h_i}-\Big(\frac{1}{h_i}+\frac{1}{h_{i+1}}\Big)u_i+\frac{u_{i+1}}{h_{i+1}}\right), && \quad \alpha=1,\\
      & \frac{\delta^{-\alpha}}{-\alpha} \left(g(x_{i}^{-})+g(x_{i}^{+})\right)-\frac{\delta^{1-\alpha}}{1-\alpha} \left(g'(x_{i}^{-})-g'(x_{i}^{+})\right)\\
      &\qquad=-\frac{\delta^{1-\alpha}}{1-\alpha} \left(\frac{u_{i-1}}{h_i}-\Big(\frac{1}{h_i}+\frac{1}{h_{i+1}}\Big)u_i+\frac{u_{i+1}}{h_{i+1}}\right), && \quad \alpha\in(1,2),
 \end{aligned}\right.
 \end{equation}
 where $g(x^{-})$ and $g(x^{+})$ denote the left and right limits of $g$ at $x$ respectively. Furthermore, by performing a series of integral manipulations, we have
 \begin{equation}\label{Local_approx:e3}
  \int_{\Omega_\delta}\frac{g(y)}{|x_i-y|^{1+\alpha}}dy 
        = \left\{  \begin{aligned}
		& \frac{\ln\delta-\ln h_i}{h_i}u_{i-1} +\left(\frac{\ln h_i}{h_i}+\frac{\ln h_{i+1}}{h_{i+1}}-\Big(\frac{1}{h_i}+\frac{1}{h_{i+1}}\Big)\ln\delta\right) u_i\\
             &\quad-\frac{\ln\delta-\ln h_{i+1}}{h_{i+1}}u_{i+1}, \quad \alpha=1,\\
		 & \frac{\delta^{1-\alpha}-h_i^{1-\alpha}}{(1-\alpha)h_i}u_{i-1}
               +\left(\frac{h_i^{-\alpha}+h_{i+1}^{-\alpha}}{1-\alpha}-\frac{\delta^{1-\alpha}}{1-\alpha}\Big(\frac{1}{h_i}+\frac{1}{h_{i+1}}\Big)\right)u_i\\
         &\quad +\frac{\delta^{1-\alpha}-h_{i+1}^{1-\alpha}}{(1-\alpha)h_{i+1}}u_{i+1}, \quad \alpha\in(0,2)\setminus\{1\}.
	\end{aligned}\right.
\end{equation}
 Thus, substituting \eqref{Local_approx:e2}--\eqref{Local_approx:e3} into \eqref{Local_approx:e1}, we obtain 
\begin{equation}\label{Lap_approx:e6}
  \mL_{loc}[u]_i 
        = \left\{ \begin{aligned}
		& -\frac{\ln h_i}{h_i}u_{i-1}+\left(\frac{\ln h_i}{h_i}+\frac{\ln h_{i+1}}{h_{i+1}}\right)u_i-\frac{\ln h_{i+1}}{h_{i+1}}u_{i+1},\quad\alpha=1,\\
		& -\frac{h_i^{-\alpha}}{1-\alpha}u_{i-1}+\frac{h_i^{-\alpha}+h_{i+1}^{-\alpha}}{(1-\alpha)} u_i-\frac{h_{i+1}^{-\alpha}}{1-\alpha}u_{i+1}, \quad \alpha\in(0,2)\setminus\{1\}.
	\end{aligned}\right.
\end{equation}

As a result, by substituting the approximations \eqref{Lap_approx:e3}, \eqref{Lap_approx:e4}, \eqref{Lap_approx:e5} and \eqref{Lap_approx:e6} of the three different parts into \eqref{Lap_approx:e2}, we derive the following fast collocation scheme for the model problem \eqref{model:nonlocal}:
\begin{equation}\label{sch:fast_Collocation}
\mL_h [U]_i=f_i:=f(x_i),\quad i=1,2,\dots,N-1; ~~ U_0=U_N=0,
\end{equation}
where $U_i$ denotes the approximations of $u_i=u(x_i)$ for $i=0,1,\dots,N$,
and the nonlocal operator $\mL_h[v]_i \approx \mL[v]_i$ is represented as follows: 
\begin{equation*} 
    \mL_h[v]_i=C_{\alpha}
     \left\{
       \begin{aligned}
        &-\frac{1}{1-\alpha}\Big(h_i^{-\alpha} v_{i-1}-\frac{h_i^{-\alpha}+h_{i+1}^{-\alpha}}{\alpha}v_i + h_{i+1}^{-\alpha} v_{i+1}\Big)\\
        &\qquad-\sum_{s=1}^{N_{e}}\theta_s \mS_{i,s}^{L}[v] -\sum_{s=1}^{N_{e}}\theta_s \mS_{i,s}^{R}[v],\quad \alpha\in(0,1),\\
        &-\frac{1+\ln h_i}{h_i} v_{i-1}+\Big(\frac{1+\ln h_i}{h_i}+\frac{1+\ln h_{i+1}}{h_{i+1}}\Big) v_i-\frac{1+\ln h_{i+1}}{h_{i+1}} v_{i+1}\\
        &\qquad+\frac{1}{\alpha}\sum_{s=1}^{\tilde{N}_{e}}\tilde{\theta}_s\tilde{ \mS}_{i,s}^{L}[v]
        +\frac{1}{\alpha}\sum_{s=1}^{\tilde{N}_{e}}\tilde{\theta}_s\tilde{ \mS}_{i,s}^{R}[v], \quad \alpha=1,\\
        &-\frac{1}{\alpha(1-\alpha)}\Big(h_i^{-\alpha}v_{i-1}-(h_i^{-\alpha}
          +h_{i+1}^{-\alpha})v_i
          +h_{i+1}^{-\alpha}v_{i+1}\Big)\\ &\qquad+\frac{1}{\alpha} \sum_{s=1}^{\tilde{N}_{e}}\tilde{\theta}_s\tilde{ \mS}_{i,s}^{L}[v]
           +\frac{1}{\alpha} \sum_{s=1}^{\tilde{N}_{e}}\tilde{\theta}_s\tilde{ \mS}_{i,s}^{R}[v],\quad \alpha\in(1,2).
	\end{aligned}
    \right.
\end{equation*}

\begin{remark} We remark that as mentioned above, if the nonlocal parts $\mL_{l}[u]$ and $\mL_{r}[u]$ are computed exactly without using such fast recursive approximations, we can derive a direct collocation scheme for model \eqref{model:nonlocal} as follows:
\begin{equation}\label{sch:direct_Collocation}
\mL [U]_i=f_i,\quad i=1,2,\dots,N-1; ~~ U_0=U_N=0.
\end{equation}
\end{remark}
\subsection{Unique solvability}
In this subsection, we discuss the existence and uniqueness of the solution to the fast collocation method \eqref{sch:fast_Collocation} for $\alpha\in(0,1)$. To this aim, we define $\bm{U}:=(U_1,U_2,\cdots,U_{N-1})^\top$ and $ \bm{f}:=(f_1,f_2,\cdots,f_{N-1})^\top$, and rewrite the developed collocation scheme \eqref{sch:fast_Collocation} into the following matrix form
\begin{equation}\label{sch:fast_Collocation:matrix}
	\bm{AU}=\bm{f},
\end{equation}
where the coefficient matrix $\bm{A}=({a}_{i, j})\in\mathbb{R}^{(N-1)\times(N-1)}$ for $\alpha\in (0,1)$ is defined by
\begin{equation}\label{matrix-form:a}
	a_{i, j}=\left\{
	\begin{aligned}
		& - \int_{x_{j-1}}^{x_{j+1}}\phi_j(y)\sum_{s=1}^{N_{e}}\theta_se^{-\lambda_s(y-x_i)}dy, & \quad j\geq i+2,\\
		& - \frac{h_{i+1}^{-\alpha}}{1-\alpha} -\int_{x_{i+1}}^{x_{i+2}} \phi_{i+1}(y)\sum_{s=1}^{N_{e}}\theta_se^{-\lambda_s(y-x_i)}dy, & \quad j= i+1,\\
		&\frac{h_i^{-\alpha}+h_{i+1}^{-\alpha}}{\alpha(1-\alpha)},& \quad j=i,\\
		& -\frac{h_i^{-\alpha}}{1-\alpha}
                    -\int_{x_{i-2}}^{x_{i-1}} \phi_{i-1}(y)\sum_{s=1}^{N_{e}}\theta_se^{-\lambda_s(x_i-y)}dy, & \quad j=i-1,\\
		& - \int_{x_{j-1}}^{x_{j+1}} \phi_j(y)\sum_{s=1}^{N_{e}}\theta_se^{-\lambda_s(x_i-y)}dy, & \quad j\leq i-2.
	\end{aligned}\right.
\end{equation}

\begin{remark} Corresponding to the direct collocation scheme \eqref{sch:direct_Collocation}, we can also get a matrix form
\begin{equation*} 
	\bm{A}^d \bm{U}=\bm{f},
\end{equation*}
where the coefficient matrix $\bm{A}^d=({a}^d_{i, j})\in\mathbb{R}^{(N-1)\times(N-1)}$ 
is defined by 
\begin{equation}\label{matrix-form:b}
	a_{i, j}^d=\left\{
	\begin{aligned}
		&\frac{h_i^{-\alpha}+h_{i+1}^{-\alpha}}{\alpha(1-\alpha)},\quad j= i,\\
        &\frac{|x_{j-1}-x_i|^{1-\alpha}-|x_{j}-x_i|^{1-\alpha}}{\alpha(1-\alpha)h_j}
          +\frac{|x_{j+1}-x_i|^{1-\alpha}-|x_{j}-x_i|^{1-\alpha}}{\alpha(1-\alpha)h_{j+1}},\quad j\neq i,
	\end{aligned}\right.
\end{equation}
In \cite{CD24}, it has been proven that on graded meshes the coefficient matrix $\bm{A}^d$ is strictly diagonally dominant, with positive entries on the main diagonal and negative off-diagonal entries. 
\end{remark}

In the following, we discuss the properties of the coefficient matrix \eqref{matrix-form:a} on general nonuniform grids, and then show the unique solvability of the collocation method.
\begin{theorem}\label{thm:dia-dom-mat}
For $\alpha\in(0,1)$, the fast collocation method \eqref{sch:fast_Collocation} is uniquely solvable on general nonuniform grids, if $\epsilon$ is sufficiently accurate.
\end{theorem}
\begin{proof} 
%
First, it follows from \eqref{matrix-form:a} and  \eqref{matrix-form:b} that 
\begin{equation}\label{matrix-form:a1}
		a_{i,i}=a_{i,i}^d=\frac{h_i^{-\alpha}+h_{i+1}^{-\alpha}}{\alpha(1-\alpha)}>0, \quad 1\leq i\leq N-1.
\end{equation}

Next, we combine \eqref{equ:soe} and \eqref{matrix-form:a} to conclude that 
	\begin{equation*}
		-\int_{x_{j-1}}^{x_{j+1}}\phi_j(y)\big(|x_i-y|^{-(1+\alpha)}+\epsilon\big)dy
             \leq a_{i, j}
             \leq 
             -\int_{x_{j-1}}^{x_{j+1}}\phi_j(y)\big(|x_i-y|^{-(1+\alpha)}-\epsilon\big)dy,
	\end{equation*}
    for $|i-j|\geq 2$,
	i.e.,
	\begin{equation}\label{matrix-form:a2}
		a_{i, j}^d-\frac{h_j+h_{j+1}}{2}\epsilon\leq a_{i, j}\leq a_{i, j}^d+\frac{h_j+h_{j+1}}{2}\epsilon.
	\end{equation}
Besides, for $j=i+1$, we have
\begin{equation*}
 \begin{aligned}
	& -\frac{h_{i+1}^{-\alpha}}{1-\alpha}
            -\int_{x_{i+1}}^{x_{i+2}}\phi_{i+1}(y)\big((y-x_i)^{-(1+\alpha)}+\epsilon\big)dy\\
		& \qquad	\leq a_{i, i+1}\leq
             -\frac{h_{i+1}^{-\alpha}}{1-\alpha}
             -\int_{x_{i+1}}^{x_{i+2}}\phi_{i+1}(y)\big((y-x_i)^{-(1+\alpha)} -\epsilon\big)dy,
		\end{aligned}	 
\end{equation*}
	i.e.,
	\begin{equation}\label{matrix-form:a3}
		a_{i,i+1}^d-\frac{h_{i+2}}{2}\epsilon\leq a_{i, i+1}\leq a_{i,i+1}^d+\frac{h_{i+2}}{2}\epsilon,
	\end{equation}
and for $j=i-1$, we have
	\begin{equation*}
		\begin{aligned}
			& -\frac{h_i^{-\alpha}}{1-\alpha} -\int_{x_{i-2}}^{x_{i-1}}\phi_{i-1}(y)\big((x_i-y)^{-(1+\alpha)}+\epsilon\big)dy\\
			& \qquad \leq a_{i, i-1}\leq-\frac{h_i^{-\alpha}}{1-\alpha} -\int_{x_{i-2}}^{x_{i-1}}\phi_{i-1}(y)\big((x_i-y)^{-(1+\alpha)}-\epsilon\big)dy,
		\end{aligned}
	\end{equation*}
	i.e.,
	\begin{equation}\label{matrix-form:a4}
		a_{i,i-1}^d-\frac{h_{i-1}}{2}\epsilon\leq a_{i, i-1}\leq a_{i,i-1}^d+\frac{h_{i-1}}{2}\epsilon.
	\end{equation}
    
	Note that for any $i\ne j$, it holds that
\begin{equation*} 
		a_{i,j}^d
        :=-\int_{x_{j-1}}^{x_{j+1}}\phi_j(y) |x_i-y|^{-(1+\alpha)} dy,      
 \end{equation*}
which implies that $a_{i, j}^d<0$ for any $i\ne j$, since the integrand is positive over the integral interval. Thus, to ensure $a_{i, j}\leq 0$ also holds for any $i\ne j$, by \eqref{matrix-form:a2}--\eqref{matrix-form:a4} it suffices to require
    $a_{i, j}^d+\frac{h_j+h_{j+1}}{2}\epsilon \le 0$,
	which leads to
		$\epsilon\leq \frac{-2 a_{i, j}^d}{h_j+h_{j+1}}$.
Furthermore, we conclude that
\begin{equation}\label{matrix-form:a5}
	\begin{aligned}
			a_{i,i}-\sum_{j=1,j\neq i}^{N-1}|a_{i, j}|&=\sum_{j=1}^{N-1}a_{i, j}\geq\sum_{j=1}^{N-1} a_{i, j}^d-(b-a)\epsilon
			=\Xi_i -(b-a)\epsilon \geq 0, 
	\end{aligned}
\end{equation}
for $\epsilon\leq \frac{\Xi_i}{b-a}$, 
    where $\Xi_i:=\frac{1}{\alpha(1-\alpha)}\Big[\frac{(x_i-a)^{1-\alpha}-(x_i-x_1)^{1-\alpha}}{h_1}+\frac{(b-x_i)^{1-\alpha}-(x_{N-1}-x_i)^{1-\alpha}}{h_{N}}\Big] $. 

Therefore, equations \eqref{matrix-form:a1}--\eqref{matrix-form:a5} imply that $\bm{A}$ is a strictly diagonally dominant matrix with positive diagonal entries and negative off-diagonal entries if 
\begin{equation}\label{solvability:condition}
\epsilon \le \min \left\{\frac{-2 a_{i,j}^d}{h_j+h_{j+1}},\frac{\Xi_i}{b-a} \right\}.
\end{equation}
Therefore, if $\epsilon$ is sufficiently small, the coefficient matrix $\bm{A}$ is invertible, and thus, the fast collocation method \eqref{sch:fast_Collocation} is uniquely solvable.
\end{proof}

\begin{remark}\label{rem:condition}
 The following calculation confirms the reliability of condition \eqref{solvability:condition} for the example considered in Section \ref{sec:num} with $(a,b)=(0,2)$. First, a simple calculation shows that
\begin{equation*}
 		\frac{\Xi_i}{b-a}\geq\frac{1}{\alpha(b-a)}\big[(x_i-a)^{-\alpha}+(b-x_i)^{-\alpha}\big]>\frac{2}{\alpha(b-a)^{1+\alpha}}>\frac{1}{2},
 	\end{equation*}
which ensures that the lower bound is uniformly larger than $1/2$. Next, if uniform grids with grid size $h:=(b-a)/N$ are considered, we have
\begin{equation*}
\begin{aligned}
\frac{-2 a_{i,j}^d}{h_j+h_{j+1}} & = \frac{1}{\alpha(1-\alpha)h^{1+\alpha}} \big[2(j-i)^{1-\alpha}-(j-i-1)^{1-\alpha}-(j-i+1)^{1-\alpha}\big]\\
&\geq \frac{1}{\alpha(1-\alpha)h^{1+\alpha}} \big[2(N-2)^{1-\alpha}-(N-3)^{1-\alpha}-(N-1)^{1-\alpha}\big]
 =:g_{\alpha}(N),
\end{aligned}
\end{equation*}
where the inequality follows from the monotonicity of the function. Figure \ref{fig:eps1} illustrates the values of $g_{\alpha}(N)$ with respect to $\alpha\in(0.001,0.999)$ for fixed $N=2^{13}$ (the finest grid used in Section \ref{sec:num}). As shown, the minimum value is about $0.25$. While, for the special symmetric graded grids \eqref{grad_mesh:e1} employed in this paper, see $\alpha=0.8$, $\kappa=2(2-\alpha)/\alpha$ and $N=2^9$, the minimum value is around $0.15$. Therefore, it suffices to choose $\epsilon$ smaller than these thresholds to ensure the unique solvability of \eqref{sch:fast_Collocation}. In practice, for example, in Section \ref{sec:num} the parameter $\epsilon$ is set to $10^{-8}$ to maintain full accuracy, which is far less than the required condition \eqref{solvability:condition}.
\begin{figure}[htbp!]\small
	\centering
	\includegraphics[width=0.5\textwidth]{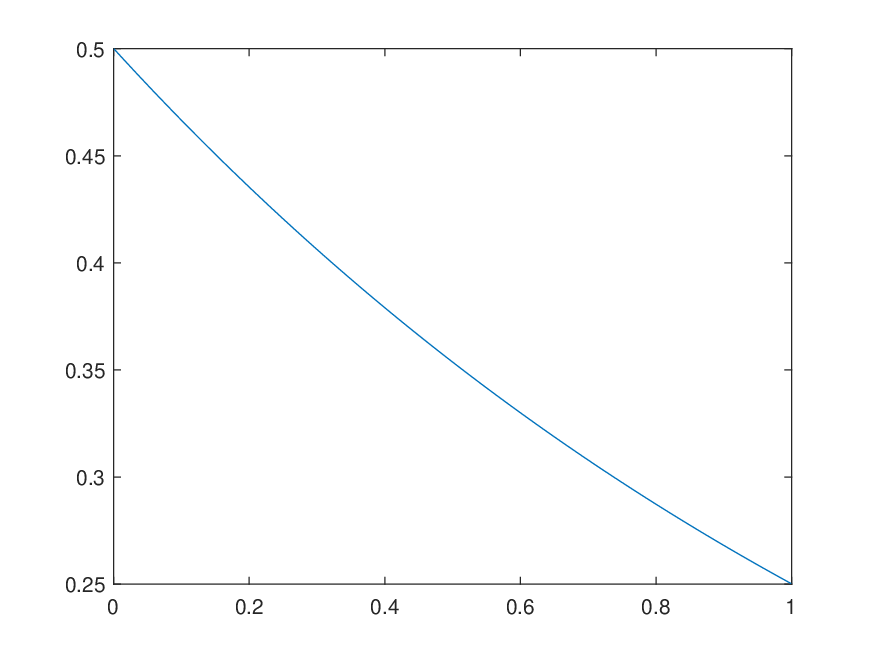}
    \caption{Values of $g_{\alpha}(N)$ with respect to $\alpha\in(0.001,0.999)$}
	\label{fig:eps1}
\end{figure}
\end{remark}

\begin{remark}
It is regrettable that the unique solvability result for the proposed fast collocation scheme with $\alpha\in[1,2)$ is not established in Theorem \ref{thm:dia-dom-mat}. The reason is that in such a case the signs of the matrix entries change and the resulting stiffness matrix no longer satisfies the diagonal dominance property. To address this issue, below we introduce a modified collocation scheme, and show that this modified scheme is uniquely solvable on uniform grids, provided that $\epsilon$ is sufficiently small.
\end{remark}

\subsection{A modified scheme}
In this subsection, we propose a modified fast collocation scheme by presenting a new local approximation for $\mL_{loc}[u]_i$ in \eqref{Lap_approx:e1}. A similar idea was considered by Zhang et al. \cite{ZG16} for hypersingular integral equations. Specifically, we consider the following approximation
\begin{equation}\label{equ:mod_loc1}
\begin{aligned}
 \int_{x_{i-1}}^{x_{i+1}} \frac{u(x_i)-u(y)}{|x_i-y|^{1+\alpha}} dy  &\approx \int_{x_{i-1}}^{x_{i+1}} \frac{u(x_i)-\left[u(x_i)+u'(x_i)(y-x_i)\right]}{|x_i-y|^{1+\alpha}}dy \\
 &=  u'(x_i) \int _{x_{i-1}}^{x_{i+1}} \frac{x_i-y}{|x_i-y|^{1+\alpha}} dy,
\end{aligned}
\end{equation}
and noting a well-known second-order approximation of $u'(x_i)$ on general grids can be given by 
\begin{equation}\label{equ:u_pri}
	u'(x_i)\approx-\frac{h_{i+1}}{h_i(h_i+h_{i+1})}u_{i-1}+\frac{h_{i+1}-h_i}{h_ih_{i+1}}u_i+\frac{h_i}{h_{i+1}(h_i+h_{i+1})}u_{i+1}.
\end{equation}
Furthermore, by taking $g(y) = x_i-y$ in \eqref{Local_approx:e1}--\eqref{Local_approx:e2} and performing a series of integral manipulations, we obtain
\begin{equation}\label{equ:mod1}
\int_{x_{i-1}}^{x_{i+1}} \frac{x_i-y}{|x_i-y|^{1+\alpha}} dy = \lim_{\delta\rightarrow 0}\Big(\int _{\Omega_{\delta}} \frac{x_i-y}{|x_i-y|^{1+\alpha}} dy + r(\delta)\Big) := \eta_i^\alpha, 
\end{equation}
where
\begin{equation*}
\eta_i^\alpha = \left\{
\begin{aligned}
& \ln h_i -\ln h_{i+1}, &&\quad \alpha = 1,\\
& \frac{h_i^{1-\alpha}-h_{i+1}^{1-\alpha}}{1-\alpha}, &&\quad \alpha \in (0,2)\setminus \{1\}.
\end{aligned}
\right.
\end{equation*}
Therefore, by substituting \eqref{equ:u_pri}--\eqref{equ:mod1} into \eqref{equ:mod_loc1}, we obtain a new local numerical approximation that replaces the local integral in \eqref{Lap_approx:e6}:
\begin{equation}\label{equ:mod_loc3}
	\mL_{loc}^{(m)}[u]_i:= 
     \Big[-\frac{h_{i+1} }{h_i (h_i+h_{i+1})} u_{i-1} + \frac{(h_{i+1}-h_i)}{h_i h_{i+1}} u_{i} + \frac{h_{i}}{h_{i+1} (h_i+h_{i+1})} u_{i+1} \Big] \eta_i^\alpha.
\end{equation}

As a result, combining \eqref{Lap_approx:e3}, \eqref{Lap_approx:e4}, \eqref{Lap_approx:e5} and \eqref{equ:mod_loc3}, we obtain the following modified fast collocation scheme 
\begin{equation}\label{sch:fast_collocation_md}
\mL^{(m)}_h [U]_i=f_i,\quad i=1,2,\dots,N-1; ~~ U_0=U_N=0,
\end{equation}
where
\begin{equation*}
    \mL^{(m)}_h[v]_i:= C_{\alpha}\left\{\begin{aligned}
    & \frac{-h_{i+1} \eta_i^\alpha }{h_i(h_i+h_{i+1})} v_{i-1}
     + \Big(\frac{h_i^{-\alpha}+h_{i+1}^{-\alpha}}{\alpha} + \frac{(h_{i+1}-h_i)\eta_i^\alpha}{h_ih_{i+1}} \Big) v_i \\
	&\qquad + \frac{h_i\eta_i^\alpha}{h_{i+1}(h_i+h_{i+1})}  v_{i+1}
     - \sum_{s=1}^{N_e} \theta_s \mS_{i,s}^{L} [v]
     - \sum_{s=1}^{N_e} \theta_s \mS_{i,s}^{R} [v],\quad \alpha \in (0,1),\\
        &\frac{-h_i - h_{i+1} - h_{i+1}\eta_i^\alpha}{h_i(h_i+h_{i+1})} v_{i-1}
         +\frac{h_i+h_{i+1}+(h_{i+1}-h_i)\eta_i^\alpha}{h_ih_{i+1}} v_i\\
         &\qquad  + \frac{-h_i-h_{i+1}+h_i\eta_i^\alpha}{(h_i+h_{i+1})h_{i+1}} v_{i+1}
         +\frac{1}{\alpha} \sum_{s=1}^{\tilde{N}_{e}}\tilde{\theta}_s\tilde{ \mS}_{i,s}^{L}[v]
          +\frac{1}{\alpha} \sum_{s=1}^{\tilde{N}_{e}}\tilde{\theta}_s\tilde{ \mS}_{i,s}^{R}[v], \quad \alpha=1,\\
        &\Big( -\frac{h_i^{-\alpha}}{\alpha(1-\alpha)} -\frac{h_{i+1}\eta_i^\alpha}{h_i(h_i+h_{i+1})} \Big) v_{i-1}
         +\Big( \frac{h_i^{-\alpha}+h_{i+1}^{-\alpha}}{\alpha} + \frac{(h_{i+1}-h_i)\eta_i^\alpha}{h_ih_{i+1}} \Big) v_i\\
	&\qquad+\Big( -\frac{h_{i+1}^{-\alpha}}{(1-\alpha)} + \frac{h_i \eta_i^\alpha}{(h_i+h_{i+1})h_{i+1}}\Big) v_{i+1} \\
     &\qquad    +\frac{1}{\alpha} \sum_{s=1}^{\tilde{N}_{e}}\tilde{\theta}_s\tilde{ \mS}_{i,s}^{L}[v]
         +\frac{1}{\alpha} \sum_{s=1}^{\tilde{N}_{e}}\tilde{\theta}_s\tilde{ \mS}_{i,s}^{R}[v],\quad \alpha\in(1,2).
	\end{aligned}
    \right.
\end{equation*}
Similarly, the modified collocation scheme \eqref{sch:fast_collocation_md} can be expressed in matrix form
\begin{equation}\label{equ:matrix-form-mod}
\bm{A}^{(m)}\bm{U}=\bm{f},
\end{equation}
where the coefficient matrix $\bm{A}^{(m)}=({a}^{(m)}_{i, j})\in\mathbb{R}^{(N-1)\times(N-1)}$ is defined by
\begin{equation}\label{matrix-form:c1}
	a_{i, j}^{(m)} = \left\{
	\begin{aligned}
		& - \int_{x_{j-1}}^{x_{j+1}}\phi_j(y)\sum_{s=1}^{N_{e}}\theta_se^{-\lambda_s(y-x_i)}dy, &\quad  j\geq i+2,\\
		& \frac{h_i \eta_i^\alpha}{h_{i+1}(h_i+h_{i+1})}  -\int_{x_{i+1}}^{x_{i+2}} \phi_{i+1}(y)\sum_{s=1}^{N_{e}}\theta_se^{-\lambda_s(y-x_i)}dy, & \quad j= i+1,\\
		& \frac{h_i^{-\alpha}+h_{i+1}^{-\alpha}}{\alpha} + \frac{(h_{i+1}-h_i)\eta_i^\alpha}{h_ih_{i+1}}, & \quad j=i,\\
		& \frac{-h_{i+1} \eta_i^\alpha }{h_i(h_i+h_{i+1})} 
            - \int_{x_{i-2}}^{x_{i-1}} \phi_{i-1}(y)\sum_{s=1}^{N_{e}}\theta_se^{-\lambda_s(x_i-y)}dy, & \quad j=i-1,\\
		& - \int_{x_{j-1}}^{x_{j+1}} \phi_j(y)\sum_{s=1}^{N_{e}}\theta_se^{-\lambda_s(x_i-y)}dy, & \quad j\leq i-2,
	\end{aligned}\right.
\end{equation}
for $\alpha\in(0,1)$, or by
\begin{equation}\label{matrix-form:c2}
	a_{i, j}^{(m)}=\left\{
	\begin{aligned}
		& -\int_{x_{j-1}}^{x_{j+1}}\phi_j'(y)\sum_{s=1}^{\tilde{N}_{e}}\tilde{\theta}_se^{-\tilde{\lambda}_s(y-x_i)}dy, & \quad j\geq i+2,\\
		& \frac{-h_i-h_{i+1}+h_i\eta_i^\alpha}{(h_i+h_{i+1})h_{i+1}} - \int_{x_{i+1}}^{x_{i+2}} \phi_{i+1}'(y) \sum_{s=1}^{\tilde{N}_{e}} \tilde{\theta}_se^{-\tilde{\lambda}_s(y-x_i)}dy, & \quad j= i+1,\\
		& \frac{h_i+h_{i+1}+(h_{i+1}-h_i)\eta_i^\alpha}{h_ih_{i+1}}, & \quad j=i,\\
		& \frac{-h_i - h_{i+1} - h_{i+1}\eta_i^\alpha}{h_i(h_i+h_{i+1})}
        + \int_{x_{i-2}}^{x_{i-1}} \phi_{i-1}'(y) \sum_{s=1}^{\tilde{N}_{e}} \tilde{\theta}_se^{-\tilde{\lambda}_s(x_i-y)}dy, & \quad j=i-1,\\
		&  \int_{x_{j-1}}^{x_{j+1}} \phi_j'(y)\sum_{s=1}^{\tilde{N}_{e}}\tilde{\theta}_se^{-\tilde{\lambda}_s(x_i-y)}dy, & \quad j\leq i-2,
	\end{aligned}\right.
\end{equation}
for $\alpha=1$, or by
\begin{equation}\label{matrix-form:c3}
	a_{i, j}^{(m)}=\left\{
	\begin{aligned}
		& -\frac{1}{\alpha} \int_{x_{j-1}}^{x_{j+1}}\phi_j'(y)\sum_{s=1}^{\tilde{N}_{e}}\tilde{\theta}_se^{-\tilde{\lambda}_s(y-x_i)}dy, & \quad j\geq i+2,\\
		& -\frac{h_{i+1}^{-\alpha}}{(1-\alpha)} + \frac{h_i \eta_i^\alpha}{(h_i+h_{i+1})h_{i+1}} - \frac{1}{\alpha} \int_{x_{i+1}}^{x_{i+2}} \phi_{i+1}'(y) \sum_{s=1}^{\tilde{N}_{e}} \tilde{\theta}_se^{-\tilde{\lambda}_s(y-x_i)}dy, & \quad j= i+1,\\
		&  \frac{h_i^{-\alpha}+h_{i+1}^{-\alpha}}{\alpha} + \frac{(h_{i+1}-h_i)\eta_i^\alpha}{h_ih_{i+1}}, & \quad j=i,\\
		& -\frac{h_i^{-\alpha}}{\alpha(1-\alpha)} -\frac{h_{i+1}\eta_i^\alpha}{h_i(h_i+h_{i+1})}
        + \frac{1}{\alpha} \int_{x_{i-2}}^{x_{i-1}} \phi_{i-1}'(y) \sum_{s=1}^{\tilde{N}_{e}} \tilde{\theta}_se^{-\tilde{\lambda}_s(x_i-y)}dy, & \quad j=i-1,\\
		& \frac{1}{\alpha} \int_{x_{j-1}}^{x_{j+1}} \phi_j'(y)\sum_{s=1}^{\tilde{N}_{e}}\tilde{\theta}_se^{-\tilde{\lambda}_s(x_i-y)}dy, & \quad j\leq i-2,
	\end{aligned}\right.
\end{equation}
for $\alpha\in(1,2)$.

\begin{lemma}\label{lem:mod-sch-sign}
For $\alpha \in (0,2) \setminus \{1\}$ and $i \geq 2$, the following inequalities hold
\begin{equation*}
\left\{\begin{aligned}
& \ln 2 - 1 < 0, \quad \ln \frac{i^2-1}{i^2} < 0 ,\\
& \frac{2^{1-\alpha} + \alpha - 2}{1-\alpha} < 0,\\
& m_i: = \frac{(i+1)^{1-\alpha} - 2i^{1-\alpha} + (i-1)^{1-\alpha}}{\alpha(1-\alpha)} < 0.
\end{aligned}\right.
\end{equation*}
\end{lemma}
\begin{proof}
The lemma can be easily proved by some basic mathematical techniques, such as monotonicity and convexity. The details are omitted for the sake of brevity.
\end{proof}

Although the modified scheme is applicable to general spatial grids, below we restrict our numerical analysis to uniform grids.
\begin{theorem}\label{thm:dia-dom-mat-m}
 The modified fast collocation scheme \eqref{sch:fast_collocation_md} is uniquely solvable on uniform grids for any $\alpha\in (0,2)$, provided that $\epsilon$ is sufficiently accurate.
\end{theorem}
\begin{proof}
We first check the signs of the entries $a_{i,j}^{(m)}$ of $\bm{A}^{(m)}$. From \eqref{matrix-form:c1}--\eqref{matrix-form:c3}, it is immediate to see that the diagonal entries satisfy 
\begin{equation*}
a_{i,i}^{(m)} = 2h^{-\alpha}/\alpha>0, \quad i=1,2,\ldots, N-1.
\end{equation*}
Moreover, analogous to \eqref{matrix-form:a2}--\eqref{matrix-form:a4}, by combining \eqref{matrix-form:c1}--\eqref{matrix-form:c3} with \eqref{equ:soe}, we deduce that the first off-diagonal entries satisfy
\begin{equation*}
\left\{\begin{aligned}
&h^{-\alpha}\frac{2^{1-\alpha}+\alpha-2}{\alpha(1-\alpha)} - \frac{h\epsilon}{2}   \leq a_{i,i\pm 1}^{(m)} \leq  h^{-\alpha}\frac{2^{1-\alpha}+\alpha-2}{\alpha(1-\alpha)} + \frac{h\epsilon}{2},\quad \alpha \in (0,1) ,\\
&\frac{\ln 2-1}{h} - \epsilon \leq a_{i,i\pm 1}^{(m)} \leq \frac{\ln 2-1}{h} + \epsilon,\quad \alpha=1,\\
&h^{-\alpha}\frac{2^{1-\alpha}+\alpha-2}{\alpha(1-\alpha)} - \frac{\epsilon}{\alpha}   \leq a_{i,i\pm 1}^{(m)} \leq  h^{-\alpha}\frac{2^{1-\alpha}+\alpha-2}{\alpha(1-\alpha)} + \frac{\epsilon}{\alpha},\quad \alpha \in (1,2),
\end{aligned}\right.
\end{equation*}
and for all other off-diagonal entries, i.e. $|i-j|\geq 2$, we have
\begin{equation*}
\left\{\begin{aligned}
& h^{-\alpha} m_{|i-j|} - h\epsilon  \leq a_{i,j}^{(m)} \leq  h^{-\alpha} m_{|i-j|} + h\epsilon,\quad \alpha \in (0,1), \\
& h^{-1}\ln\frac{(i-j)^2-1}{(i-j)^2} - 2\epsilon \leq a_{i,j}^{(m)} \leq h^{-1}\ln\frac{(i-j)^2-1}{(i-j)^2} + 2\epsilon,\quad \alpha=1,\\
& h^{-\alpha} m_{|i-j|} - \frac{2\epsilon}{\alpha}  \leq a_{i,j}^{(m)} \leq  h^{-\alpha} m_{|i-j|} + \frac{2\epsilon}{\alpha},\quad \alpha \in (1,2).
\end{aligned}\right.
\end{equation*}
Thus, by Lemma \ref{lem:mod-sch-sign} and above discussions, if the SOE tolerance is sufficiently accurate, i.e. 
\begin{equation}\label{unique:condition:md:e1}
\epsilon \leq \left\{ \begin{aligned}
 &h^{-1-\alpha} \min\Big\{\frac{2(2^{1-\alpha}+\alpha-2)}{\alpha(\alpha-1)}, -m_{|i-j|} \Big\},\quad \alpha\in(0,1),\\
& h^{-1} \min\Big\{1-\ln 2, \frac{1}{2} \ln \frac{(i-j)^2}{(i-j)^2-1} \Big\},\quad \alpha=1,\\
& h^{-\alpha} \min \Big\{\frac{2^{1-\alpha}+\alpha-2}{\alpha-1}, \frac{-\alpha m_{|i-j|}}{2} \Big\},\quad \alpha\in(1,2),
\end{aligned}\right.
\end{equation}
then all diagonal entries of $\bm{A}^{(m)}$ are positive,  whereas all non-diagonal entries are negative.

Next, we verify the diagonal dominance of the coefficient matrix $\bm{A}^{(m)}$. For $\alpha \in (0,1)$, a simple calculation shows
\begin{equation*}
\begin{aligned}
& a_{i,i}^{(m)} - \sum_{j=1,j\neq i}^{N-1} |a_{i,j}^{(m)}|= \sum_{j=1}^{N-1} a_{i,j}^{(m)}  \\
& \ge h^{-\alpha}\left\{\begin{aligned}
&\frac{2}{\alpha}+\frac{2^{1-\alpha}+\alpha-2}{\alpha(1-\alpha)} + \sum_{j=2}^{N-2} m_j - (\frac{1}{2}+N-3)h^{1+\alpha}\epsilon, \quad i = 1, N-1,\\
& \frac{2}{\alpha}+\frac{2(2^{1-\alpha}+\alpha-2)}{\alpha(1-\alpha)} + \sum_{j=1}^{i-2} m_{i-j} + \sum_{j=i+2}^{N-1} m_{j-i} - (N-3)h^{1+\alpha}\epsilon, \quad 2 \leq i \leq N-1
\end{aligned}\right.\\
& \geq \frac{h^{-\alpha}}{\alpha(1-\alpha)} 
\left\{\begin{aligned}
& 1-\alpha + (N-1)^{1-\alpha}-(N-2)^{1-\alpha} - (N-\frac{5}{2})h^{1+\alpha}\alpha(1-\alpha)\epsilon, \quad i=1,\,N-1,\\
& \sum_{\kappa=i, N-i} \left[\kappa^{1-\alpha} - (\kappa-1)^{1-\alpha} \right]
 - (N-3)h^{1+\alpha}\alpha(1-\alpha)\epsilon,\quad 2\leq i\leq N-2.
\end{aligned}\right.
\end{aligned}
\end{equation*}
Similarly, for $\alpha=1$, it holds
\begin{equation*}
\begin{aligned}
& a_{i,i}^{(m)} - \sum_{j=1,j\neq i}^{N-1} |a_{i,j}^{(m)}| = \sum_{j=1}^{N-1} a_{i,j}^{(m)}\\
& \quad \geq h^{-1} 
\left\{\begin{aligned}
& 1 + \ln\frac{N-1}{N-2} - (2N-5)h\epsilon,\quad i=1,\, N-1,\\
& \sum_{\kappa=i, N-i} \ln \frac{\kappa}{\kappa-1}  - (2N-6)h\epsilon,\quad 2\leq i\leq N-2,
\end{aligned}\right.
\end{aligned}
\end{equation*}
and for $\alpha \in (1,2)$, we have
\begin{equation*}
\begin{aligned}
& a_{i,i}^{(m)} - \sum_{j=1,j\neq i}^{N-1} |a_{i,j}^{(m)}|= \sum_{j=1}^{N-1} a_{i,j}^{(m)}\\
& \quad \ge \frac{h^{-\alpha}}{\alpha(1-\alpha)} 
\left\{\begin{aligned}
& 1-\alpha + (N-1)^{1-\alpha}-(N-2)^{1-\alpha} - (2N-5)h^\alpha(1-\alpha)\epsilon, \quad i=1,\,N-1,\\
& \sum_{\kappa=i, N-i} \left[\kappa^{1-\alpha} - (\kappa-1)^{1-\alpha} \right]
 - (2N-6)h^\alpha(1-\alpha)\epsilon,\quad 2\leq i\leq N-2.
\end{aligned}\right.
\end{aligned}
\end{equation*}
Thus, if $\epsilon$ further satisfies 
\begin{equation}\label{unique:condition:md:e2}
\epsilon \leq \frac{1}{2h^{\alpha-1}}\left\{\begin{aligned}
& \frac{2}{\alpha h}\min \Big\{1 + \frac{(N-1)^{1-\alpha}-(N-2)^{1-\alpha}}{1-\alpha},\sum_{\kappa=i, N-i}\frac{  \kappa^{1-\alpha} - (\kappa-1)^{1-\alpha}}{1-\alpha} \Big\},\,\alpha \in (0,1),\\
& \min \Big\{1+\ln \frac{N-1}{N-2},\,\sum_{\kappa=i, N-i} \ln \frac{\kappa}{\kappa-1}\Big\}, \quad \alpha=1,\\
& \min \Big\{1 + \frac{(N-1)^{1-\alpha}-(N-2)^{1-\alpha}}{1-\alpha},\sum_{\kappa=i, N-i}\frac{  \kappa^{1-\alpha} - (\kappa-1)^{1-\alpha}}{1-\alpha} \Big\},\,\alpha \in (1,2),
\end{aligned}\right.
\end{equation}
the coefficient matrix $\bm{A}^{(m)}$ is strictly diagonally dominant, which guarantees that the modified fast collocation scheme is uniquely solvable.
\end{proof}
\begin{remark} Analogous to Remark \ref{rem:condition}, for the example considered in Section \ref{sec:num} with $N=2^{13}$, we can numerically verify that the conditions \eqref{unique:condition:md:e1}--\eqref{unique:condition:md:e2} require 
$
\epsilon \le 2.5004\times 10^{-4}
$ for $\alpha\in(0,1)$,
$
\epsilon \le 3.0532\times 10^{-5}
$ for $\alpha=1$ and $
\epsilon \le 3.0542\times 10^{-5} 
$
for $\alpha\in (1,2)$. All cases can be automatically satisfied in our simulation.
\end{remark}

\subsection{Efficient implementation}
In practice, the coefficient matrices $\bm{A}$ and $\bm{A}^{(m)}$ associated with the developed collocation schemes do not need to be assembled explicitly. Instead, based on any iterative Krylov subspace solver, such as the BiCGSTAB method (Algorithm \ref{alg:bicg}), we can develop a fast version BiCGSTAB algorithm by performing efficient matrix-vector multiplications without matrix assembly.
\begin{algorithm}[!ht] 
	\caption{The standard BiCGSTAB method for $\bm{Ag}=\bm{f}$ \cite{BB94}}\label{alg:bicg}
	\begin{algorithmic}[1]
		\STATE Compute $\bm{r}^{(0)}=\bm{f}-\bm{Ag}^{(0)}$ for some initial guess $\bm{g}^{(0)}$
		\STATE Choose $\tilde{\bm{r}}$ (for example, $\tilde{\bm{r}}=\bm{r}^{(0)}$)
		\FOR {$i=1,2,\cdots$}
		\STATE $\rho_{i-1}=\tilde{\bm{r}}^\top \bm{r}^{(i-1)}$
		\STATE {\bf if} $\rho_{i-1}=0$ method fails
		\IF {$i=1$}
		\STATE $\bm{p}^{(i)}=\bm{r}^{(i-1)}$
		\ELSE
		\STATE $\beta_{i-1}=\left(\rho_{i-1}/\rho_{i-2}\right)(\alpha_{i-1}/\omega_{i-1})$
		\STATE $\bm{p}^{(i)}=\bm{r}^{(i-1)}+\beta_{i-1}(\bm{p}^{(i-1)}-\omega_{i-1}\bm{\mu}^{(i-1)})$
		\ENDIF
		\STATE compute $\bm{\mu}^{(i)}=\bm{Ap}^{(i)}$
		\STATE $\alpha_i=\rho_{i-1}/\tilde{\bm{r}}^\top\bm{\mu}^{(i)}$
		\STATE $\bm{s}=\bm{r}^{(i-1)}-\alpha_i\bm{\mu}^{(i)}$
		\STATE check norm of $\bm{s}$; if small enough: set $\bm{g}^{(i)}=\bm{g}^{(i-1)}+\alpha_i\bm{p}^{(i)}$ and stop
		\STATE comput $\bm{t}=\bm{As}$
		\STATE $\omega_i=\bm{t}^\top \bm{s}/\bm{t}^\top \bm{t}$
		\STATE $\bm{g}^{(i)}=\bm{g}^{(i-1)}+\alpha_i\bm{p}^{(i)}+\omega_i\bm{s}$
		\STATE $\bm{r}^{(i)}=\bm{s}-\omega_i\bm{t}$
		\STATE check convergence; continue if necessary
		\STATE for continuation it is necessary that $\omega_i\neq0$
		\ENDFOR
	\end{algorithmic}
\end{algorithm}

\begin{theorem}\label{thm:operation} For $\alpha \in (0,1)$, the matrix-vector multiplication $\bm{Ag}$ for any vector $\bm{g}$  can be carried out in $\mO(NN_{e})$ operations. Moreover, the total memory requirement is also of order $\mO(NN_{e})$. 
\end{theorem}
\begin{proof}
  On the one hand, the terms $ \mS_{i,s}^{L}[\bm{g}]$ and $ \mS_{i,s}^{R}[\bm{g}]$ can be efficiently computed using the recurrence relations \eqref{Lap_approx:e3a} and \eqref{Lap_approx:e5a} with only $\mO(1)$ operations per evaluation, which leads to an overall computational complexity of order $\mO(N)$ for each $s$. Consequently, the computational cost for the fast approximations of $\mL_{l}[\bm{g}]$ and $\mL_{r}[\bm{g}]$ is $\mO(NN_{e})$ in total. On the other hand, the local part $\mL_{loc}[\bm{g}]$ incurs a computational cost of $\mO(N)$. In summary, the matrix-vector multiplication $\bm{Ag}$ can be evaluated with $\mO(NN_{e})$ operations.

 It is obvious that the stiffness matrix $\bm{A}$ has a complicated structure. However, we do not need to generate it explicitly. Instead, the algorithm can be performed in $\mO(NN_{e})$ memory. Specifically, the storage of parameters $\theta_s$ for $1\leq s\leq N_{e}$ (resp. $\tilde{\theta}_s$ for $1\leq s\leq \tilde{N}_{e}$) requires $\mO(N_{e})$ memory, while the coefficients $\{\omega_{i,s}\}$, $\{\mu_{i,s}^{\pm}\}$ and $\{\nu_{i,s}^{\pm}\}$ for $1\leq i\leq N$, $1\leq s\leq N_{e}$  contribute $\mO(NN_{e})$ memory. 
   In addition, the input vector $\bm{g}$ needs $\mO(N)$ memory. Finally, the recurrence variables $ \mS_{i,s}^{L}[\bm{g}]$ and $ \mS_{i,s}^{R}[\bm{g}]$ can be updated in $\mO(NN_e)$ memory for all $i$ and $s$, which contribute to $\mO(NN_e)$ memory for the approximations $\mL_{l}[\bm{g}]$ and $\mL_{r}[\bm{g}]$. Meanwhile, the local approximation $\mL_{loc}[\bm{g}]$ in \eqref{Lap_approx:e6} requires $\mO(N)$ memory. Therefore, the total memory requirement amounts to $\mO(N_{e})+\mO(NN_{e})+\mO(N)=\mO(NN_{e})$.
 \end{proof}

Similarly, for $\alpha\in[1,2)$ we have the same conclusion.
\begin{theorem}\label{thm:memory}
For $\alpha \in [1,2)$, the computational complexity and memory requirement of matrix-vector multiplication $\bm{Ag}$ for any vector $\bm{g}$ are $\mO(N\tilde{N}_{e})$.
 \end{theorem}

Based on the above discussions, we propose Algorithm \ref{alg:mv} to efficiently perform the matrix-vector multiplication $\bm{Ag}$ that arises in the BiCGSTAB method (see Algorithm \ref{alg:bicg}) for given vector $\bm{g}$. 
\begin{algorithm}
    \caption{Efficient matrix-vector multiplication $\bm{Ag}$ for any vector $\bm{g}$}
    \begin{algorithmic}[1]\label{alg:mv}
        \REQUIRE the fractional order $\alpha\in(0,1)$ (resp. $\alpha\in[1,2)$), the weights $\bm{\theta}:=\{\theta_s\}$ (resp. $\bm{\tilde{\theta}}:=\{\tilde{\theta}_s\}$) of the SOE approximation, the grid points $\{x_{i}\}$, the coefficients $\{\omega_{i,s},\mu_{i,s}^{L,R},\nu_{i,s}^{L,R}\}$ (resp. $\{\tilde{\omega}_{i,s},\tilde{\mu}_{i,s}^{L,R}\}$) and vector  $\bm{g}$
        \ENSURE the matrix-vector multiplication $\bm{Ag}$ 
        \STATE Set $N=\text{length}(\bm{g})$, $N_{e}=\text{length}(\bm{\theta})$ (resp.  $\tilde{N}_{e}=\text{length}(\tilde{\bm{\theta}})$), $\bm{\mathcal{S}}^L:=\{\mathcal{S}_{i,s}^L\}=\text{zeros}(N,N_e)$, $\bm{\mathcal{S}}^R:=\{\mathcal{S}_{i,s}^R\}=\text{zeros}(N,N_e)$ (resp. $\bm{\tilde{\mathcal{S}}}^L:=\{\tilde{\mathcal{S}}_{i,s}^L\}=\text{zeros}(N,\tilde{N}_e)$, $\bm{\tilde{\mathcal{S}}}^R:=\{\tilde{\mathcal{S}}_{i,s}^R\}=\text{zeros}(N,\tilde{N}_e)$)
        \FOR {$s=1:N_{e}$ (resp. $\tilde{N}_{e}$)}
        \FOR {$i=2:N$}
		\STATE Compute $\mS_{i,s}^{L}$ (resp. $\tilde{\mS}_{i,s}^{L}$) by the recursive formula \eqref{Lap_approx:e3a} (resp. \eqref{Lap_approx:e4a}) with $\mS_{1,s}^{L}=0$ (resp. $\tilde{\mS}_{1,s}^{L}=0$ ) 
        \ENDFOR
        \FOR {$i=N-1:1$}
		\STATE Compute $\mS_{i,s}^{R}$ (resp. $\tilde{\mS}_{i,s}^{R}$) by the recursive formula \eqref{Lap_approx:e5a} with $\mS_{N,s}^{R}=0$ (resp. $\tilde{\mS}_{N,s}^{R}=0$)
        \ENDFOR
		\ENDFOR
        \STATE Compute the left-sided approximation $\mL_l[\bm{g}]$ by formula \eqref{Lap_approx:e3} (resp. \eqref{Lap_approx:e4})
        \STATE Compute the right-sided approximation $\mL_r[\bm{g}]$ by formula \eqref{Lap_approx:e5}
        \STATE Compute the local part $\mL_{loc}[\bm{g}]$ by formula \eqref{Lap_approx:e6}
        \STATE Assemble the matrix-vector multiplication $\bm{Ag}:=\mL_l[\bm{g}]+\mL_r[\bm{g}]+\mL_{loc}[\bm{g}]$
    \end{algorithmic}
\end{algorithm}

\begin{remark}
Note that $N_{e}$ or $\tilde{N}_{e}$ is usually of order $\mO(\log^2N)$ as established in Lemma \ref{lem:soe}. Consequently, the total computational cost and memory requirement for the matrix-vector multiplication are reduced from $\mO(N^2)$ to $\mO(N\log^2N)$. This results in a significant improvement in computational efficiency especially for large-scale simulations. Moreover, the proposed fast algorithm is applicable to general nonuniform grids.
\end{remark}

\begin{remark}
Since the modified fast collocation scheme \eqref{sch:fast_collocation_md} only alters the local part approximation compared to the original fast collocation scheme \eqref{sch:fast_Collocation}, therefore the conclusions of Theorems \ref{thm:operation} and \ref{thm:memory} remain valid, i.e., the matrix-vector multiplication $\bm{A}^{(m)}\bm{g}$ for any vector $\bm{g}$ can be carried out in $\mO(NN_e)$ operations for $\alpha\in(0,1)$ (resp. $\mO(N\tilde{N}_e)$ operations for $\alpha\in[1,2)$) with a memory cost of the same order. Meanwhile, Algorithm \ref{alg:mv} can be readily extended to handle the matrix-vector multiplication $\bm{A}^{(m)}\bm{g}$ just with \eqref{Lap_approx:e6} replaced by \eqref{equ:mod_loc3}. 
\end{remark}

\begin{remark} It should be noted that the condition numbers of the coefficient matrices in \eqref{sch:fast_Collocation:matrix} and \eqref{equ:matrix-form-mod} may deteriorate in large-scale simulations, consequently resulting in a sharp increase in iteration counts. Thus, to further improve computational efficiency, a banded preconditioner is introduced. For the fast collocation scheme \eqref{sch:fast_Collocation}, Theorem \ref{thm:dia-dom-mat} establishes that the corresponding stiffness matrix $\bm{A}$ are strictly diagonally dominant, and the discussions in \eqref{matrix-form:a1}--\eqref{matrix-form:a4} show that 
\begin{equation*}
a_{ii}=a^d_{ii},\quad \text{and} \quad
|a_{i, j}-a^d_{i, j}|\leq\left\{\begin{aligned}
&\frac{h_{i-1}}{2}\epsilon, &\quad j=i-1,\\
&\frac{h_{i+2}}{2}\epsilon, &\quad j=i+1,\\
&\frac{h_{j+1}+h_j}{2}\epsilon, &\quad |j-i|\geq 2.
\end{aligned}\right.
\end{equation*} 
Exploiting this property, a banded preconditioner $\bm{P}$ with a bandwidth of $2l-1$ ($l \geq 1$) can be constructed from the direct stiffness matrix $\bm{A}^d$. It is easy to see that the preconditioner $\bm{P}$ remains diagonally dominant, and as $\epsilon$ approaches zero, $\bm{P}^{-1}$ shall approach $\bm{A}^{-1}$.  The same reasoning extends to the modified fast collocation scheme \eqref{sch:fast_collocation_md}, where the stiffness matrix $\bm{A}^{(m)}$ also preserves diagonal dominance (see Theorem \ref{thm:dia-dom-mat-m}), and the corresponding banded preconditioner can thus be constructed analogously. The proposed preconditioner shall effectively improve the condition number, leading to fewer iterations and thus higher computational efficiency. Numerical experiments presented in Section \ref{sec:num} further confirm the effectiveness of the proposed preconditioner.
\end{remark}

\section{Error analysis}\label{sec:ana}

In this section, we restrict our error analysis to the case $\alpha\in(0,1)$. However, the case $\alpha\in[1,2)$ requires much more involved, which will be explored in future work. Without loss of generality, we take $(a,b)=(0,2L)$, and given an even integer $N$, we consider the following symmetric graded spatial grids
\begin{equation}\label{grad_mesh:e1}
	x_i=\left\{\begin{aligned}
		&L\Big(\frac{2i}{N}\Big)^\kappa, &i=0,1,\cdots,N/2,\\
		&2L-L\Big(2-\frac{2i}{N}\Big)^\kappa=2L-x_{N-i},&i=N/2+1,\cdots, N,
	\end{aligned}\right.
\end{equation}
where the constant $\kappa\geq 1$ is a user-defined grading parameter. In particular, the case $\kappa=1$ reduces to the uniform spatial partition. Throughout the paper, we use $A \lesssim B$ to represent there is a positive constant $c$ such that $A \le c B$. By the mean value theorem, it is easy to verify that the grid size $h_i=x_i-x_{i-1}$ satisfies 
\begin{equation}\label{grad_mesh:e2}
	h_i\lesssim 
        \left\{
	\begin{aligned}
		& N^{-\kappa}i^{\kappa-1},    &i = 1,2,\dots, N/2,\\
		& N^{-\kappa}(N+1-i)^{\kappa-1}, & i=N/2+1,\cdots, N,
	\end{aligned}\right.
    \quad \max_{1\le i \le N} h_i \lesssim N^{-1},
\end{equation}
which will be used frequently in the subsequent analysis. 

Motivated by \eqref{reg:ass}, we adopt a more general and weaker regularity assumption:
\begin{equation}\label{model:regu}
	\Big|\frac{\partial^{\ell}}{\partial x^{\ell}} u(x)\Big|\lesssim[(x-a)(b-x)]^{\sigma-\ell},  ~\ell=0, 1, 2,
\end{equation}
where $\sigma\in (0,\frac{\alpha}{2}]$ is a fixed parameter.

In the following, we shall present some preliminary lemmas. First, we give an estimate for the left-sided nonlocal integral error of the interpolation.
\begin{lemma}\label{lem:err_hisL}
	Assume that \eqref{model:regu} holds. Then, for $1\leq i\le N/2$, we have
    \begin{equation*}
		\Big| \int_{-\infty}^{x_{i-1}} \frac{\Pi_hu(y)-u(y)}{(x_{i}-y)^{1+\alpha}} dy \Big|
		\lesssim \left\{\begin{aligned}
		 & N^{\kappa(\alpha-\sigma)} i^{-\min\{\kappa(1+\alpha), \, \kappa(\alpha-\sigma)+2-\alpha\}}, & \kappa(1+\sigma)<2,\\
		 & N^{\kappa(\alpha-\sigma)} i^{-\kappa(1+\alpha)} \max\{\ln i,i^\alpha\}, & \kappa(1+\sigma)=2,\\
		 & N^{\kappa(\alpha-\sigma)} i^{-\kappa(\alpha-\sigma)-2+\alpha}, & \kappa(1+\sigma)>2.
		 \end{aligned}\right.
	\end{equation*}	
\end{lemma}
\begin{proof}	Due to the zero boundary condition of \eqref{model:nonlocal}, we decompose the integral into two parts
\begin{equation}\label{err_hisL:e1}
	\begin{aligned}
		\int_{-\infty}^{x_{i-1}} \frac{\Pi_hu(y)-u(y)}{(x_i-y)^{1+\alpha}} dy 
		& =  \int_{x_0=0}^{x_1} \frac{\Pi_hu(y)-u(y)}{(x_i-y)^{1+\alpha}} dy   + \sum_{j=1}^{i-2}\int_{x_j}^{x_{j+1}} \frac{\Pi_hu(y)-u(y)}{(x_i-y)^{1+\alpha}} dy \\
		&=: I_1+I_2.
	\end{aligned}
\end{equation}

We address the term $I_1$ by considering the cases $i=1$ and $i>1$ separately. First, for $i=1$, the integration by parts formula gives
\begin{equation*}
    \begin{aligned}
		|I_1| & =\Big| -\frac{1}{\alpha} \int_{0}^{x_1} (x_1-y)^{-\alpha} \big(\Pi_hu(y)-u(y)\big)' dy \Big|\\
            & \leq \frac{1}{\alpha}\, \Big| \int_{0}^{x_1} (x_1-y)^{-\alpha} \frac{u(x_1)-u(x_0)}{x_1-x_0} dy \Big|  
            +  \frac{1}{\alpha}\,  \int_{0}^{x_1} (x_1-y)^{-\alpha} |u'(y)| dy\\
            & \le \frac{x_1^{-\alpha}}{\alpha(1-\alpha)}  \int_{0}^{x_1} |u'(y)|dy   
             +  \frac{1}{\alpha}\,  \int_{0}^{x_1} (x_1-y)^{-\alpha} |u'(y)| dy.    \end{aligned}
	\end{equation*}
Note that assumption \eqref{model:regu} implies that $|u'(y)|\lesssim y^{\sigma-1}$ for $y\in (0, x_1)$. Substituting this bound into the expression for $I_1$ and applying the change of variables $y= x_1 t$ to the second integral, we obtain
    \begin{equation*}
    \begin{aligned}
    |I_1|  & \lesssim \frac{x_1^{-\alpha}}{\alpha(1-\alpha)}  \int_{0}^{x_1} y^{\sigma-1}dy  
             +   \frac{1}{\alpha} \int_{0}^{x_1} (x_1-y)^{-\alpha} y^{\sigma-1} dy\\
         & \lesssim \frac{x_1^{\sigma-\alpha}}{\sigma\alpha(1-\alpha)} + \frac{x_1^{\sigma-\alpha}}{\alpha}\int_0^1 (1-t)^{-\alpha}t^{\sigma-1}dt\\
         & \lesssim \frac{1}{\alpha} \big[ \frac{1}{\sigma(1-\alpha)} + B(1-\alpha,\sigma) \big] x_1^{\sigma-\alpha},
             \end{aligned}
    \end{equation*}
    where $B(1-\alpha,\sigma)$ denotes the Beta function and is finite for $\sigma>0$ and $\alpha<1$. Then using \eqref{grad_mesh:e1}, we obtain
    \begin{equation}\label{err_hisL:e2}
    |I_1|  \lesssim x_1^{\sigma-\alpha} \lesssim N^{\kappa(\alpha-\sigma)}.
    \end{equation}
    
	Next, for the case $2 \le i\le N/2$, we use the definition of $\Pi_hu(y)$ to rewrite
	\begin{equation*}
	\begin{aligned}
		|I_1| & = \Big| \int_{0}^{x_1} (x_i-y)^{-1-\alpha} \Big[ \frac{y-x_0}{h_1} \int_{y}^{x_1} u'(s)ds - \frac{x_1-y}{h_1} \int_{0}^{y} u'(s)ds \Big] dy \Big|\\
		&\le \int_{0}^{x_1} (x_i-y)^{-1-\alpha} dy \int_{0}^{x_1} |u'(s)|ds 
        \lesssim x_1^{\sigma}\int_{0}^{x_1} (x_i-y)^{-1-\alpha} dy.
	\end{aligned}
	\end{equation*}
By applying the mean value theorem for integrals and using the monotonicity and positivity of the function $x^{-1-\alpha}$ over $[x_j, x_{j+1}]$, we obtain
    \begin{equation}\label{equ:er_hisl:e1}
    \int_{x_j}^{x_{j+1}}(x_i-y)^{-1-\alpha}dy \leq h_{j+1}(x_i-x_{j+1})^{-1-\alpha},\quad \text{for}~ j <i-1.
    \end{equation}
It then follows from \eqref{grad_mesh:e1}--\eqref{grad_mesh:e2} and \eqref{equ:er_hisl:e1} that
	\begin{equation}\label{err_hisL:e3}
			|I_1|	\lesssim x_1^{\sigma} h_1 (x_i-x_1)^{-1-\alpha} 
            \lesssim h_1^{1+\sigma}  x_i^{-1-\alpha}  \lesssim N^{\kappa(\alpha-\sigma)} i^{-\kappa(1+\alpha)}, 
	\end{equation}
where we have used the fact that $\frac{1}{2} x_i \le x_i-x_1 \le x_i$ for $i=2,\cdots,N/2$. 
	
We now turn to estimate $I_2$. By the standard linear interpolation error estimate  and the solution regularity \eqref{model:regu}, we have
	\begin{equation}\label{err_hisL:e4}
    \begin{aligned}
		|I_2|
          & \lesssim \sum_{j=1}^{i-2} h_{j+1}^2 \max_{\xi\in[x_j,x_{j+1}]}|u''(\xi)| \int_{x_j}^{x_{j+1}}(x_i-y)^{-1-\alpha}dy \\
	    &\lesssim \sum_{j=1}^{i-2} h_{j+1}^2 x_j^{\sigma-2} \int_{x_j}^{x_{j+1}}(x_i-y)^{-1-\alpha}dy \\
		&= \bigg[\sum_{j=1}^{\lceil i/2 \rceil-1}  +   \sum_{j=\lceil i/2 \rceil}^{i-2}\bigg] h_{j+1}^2x_j^{\sigma-2}\int_{x_j}^{x_{j+1}}(x_i-y)^{-1-\alpha}dy 
		=: I_{21}+I_{22},
	\end{aligned}
	\end{equation}
	where $\lceil i/2 \rceil$ represents the smallest integer that larger than ${i}/{2}$. 
    
    For $I_{21}$, we apply \eqref{equ:er_hisl:e1} and \eqref{grad_mesh:e1}--\eqref{grad_mesh:e2} to obtain
	\begin{equation*}
	\begin{aligned}
		I_{21} & \lesssim \sum_{j=1}^{\lceil i/2 \rceil-1} h_{j+1}^3 x_j^{\sigma-2} (x_i-x_{j+1})^{-1-\alpha}\\
		& \lesssim N^{\kappa(\alpha-\sigma)} \sum_{j=1}^{\lceil i/2 \rceil-1} j^{\kappa(1+\sigma)-3}(i^k-(j+1)^\kappa)^{-1-\alpha}\\ 
        & \lesssim N^{\kappa(\alpha-\sigma)} i^{-\kappa(1+\alpha)} \sum_{j=1}^{\lceil i/2 \rceil-1}j^{\kappa(1+\sigma)-3}.
	\end{aligned}
	\end{equation*}
	 Here, by monotonicity, the fact that
      \begin{equation}\label{err_hisL:e4c}
      \begin{aligned}
          \big(i^\kappa-(j+1)^\kappa\big)^{-1-\alpha} 
          & = i^{-\kappa(1+\alpha)} \Big[1-\left(\frac{j+1}{i}\right)^\kappa\Big]^{-1-\alpha}\\
          & \leq i^{-\kappa(1+\alpha)}\big[1-(\frac{1}{2})^\kappa\big]^{-1-\alpha} 
             \lesssim i^{-\kappa(1+\alpha)},\quad \text{for}~ j \leq \lceil i/2 \rceil -1
    \end{aligned}
    \end{equation}
     has been applied in the last step. 
We next estimate the summation $\sum_{j=1}^{\lceil i/2 \rceil-1}j^{\kappa(1+\sigma)-3}$ that depends on the value of $\kappa(1+\sigma)-3$ as follows:
    \begin{itemize}
    \item Case 1, $\kappa(1+\sigma)-3<-1$. As known, in this case the positive infinite series $\sum_{j=1}^{ +\infty}j^{\kappa(1+\sigma)-3}$ converges, so the partial sum is uniformly bounded:
    \begin{equation*}
    \sum_{j=1}^{\lceil i/2 \rceil-1}j^{\kappa(1+\sigma)-3} \leq \sum_{j=1}^{ +\infty}j^{\kappa(1+\sigma)-3} \lesssim 1.
    \end{equation*}
    \item Case 2, $\kappa(1+\sigma)-3=-1$. In this case, the summation becomes a partial harmonic series, thus
     \begin{equation*}
    \sum_{j=1}^{\lceil i/2 \rceil-1}j^{\kappa(1+\sigma)-3} = \sum_{j=1}^{ \lceil i/2 \rceil-1}j^{-1} \lesssim \ln i.
    \end{equation*}
    \item Case 3, $\kappa(1+\sigma)-3 >-1$. This case is further divided into three parts.
    \begin{itemize}
    \item Part I, $-1<\kappa(1+\sigma)-3<0$.  Since $f(x)=x^{\kappa(1+\sigma)-3}$ is positive and decreasing on $x \ge 0$. Hence, for each $j\ge 1$, $f(j) \leq \int_{j-1}^{j} f(x) dx$ and thus
    \begin{equation*}
    \sum_{j=1}^{\lceil i/2 \rceil-1}j^{\kappa(1+\sigma)-3} \le \sum_{j=1}^{i}f(j)  \leq \int_0^{i} x^{\kappa(1+\sigma)-3}dx \lesssim i^{\kappa(1+\sigma)-2}.
    \end{equation*}
    \item Part II, $\kappa(1+\sigma)-3=0$. We directly have
    \begin{equation*}
    \sum_{j=1}^{\lceil i/2 \rceil-1}j^{\kappa(1+\sigma)-3} = \sum_{j=1}^{\lceil i/2 \rceil-1} 1 \lesssim i.
    \end{equation*}
    \item Part III, $\kappa(1+\sigma)-3>0$. Here $f(x)=x^{\kappa(1+\sigma)-3}$ is positive and increasing on $x \ge 0$. Hence, for each $j \ge 0$, $f(j) \leq \int_{j}^{j+1} f(x) dx$ and thus 
    \begin{equation*}
    \sum_{j=1}^{\lceil i/2 \rceil-1}j^{\kappa(1+\sigma)-3} \le \sum_{j=0}^{i-1} f(j)  \leq \int_0^{i} x^{\kappa(1+\sigma)-3}dx \lesssim i^{\kappa(1+\sigma)-2}.
    \end{equation*}
    \end{itemize}
    In summary, in this case we have
        \begin{equation*}
    \sum_{j=1}^{\lceil i/2 \rceil-1}j^{\kappa(1+\sigma)-3}  \lesssim i^{\kappa(1+\sigma)-2}.
    \end{equation*}
    \end{itemize}
    Combining all three cases, we finally arrive at the estimate for $I_{21} $ as follows:
	\begin{equation}\label{err_hisL:e4a}
			I_{21} \lesssim \left\{
			\begin{aligned}
				& N^{\kappa(\alpha-\sigma)} i^{-\kappa(1+\alpha)}, & \text{if} \quad \kappa(1+\sigma)<2,\\ 
				& N^{\kappa(\alpha-\sigma)} i^{-\kappa(1+\alpha)}\ln i, & \text{if} \quad \kappa(1+\sigma)=2, \\
				& N^{\kappa(\alpha-\sigma)} i^{-\kappa(\alpha-\sigma)-2}, & \text{if} \quad \kappa(1+\sigma)>2.
		\end{aligned}\right.
	\end{equation}
    
On the other hand, for $I_{22}$, we directly use \eqref{grad_mesh:e1}--\eqref{grad_mesh:e2} to obtain
	\begin{equation*}
	\begin{aligned}
		I_{22} & \lesssim  \sum_{j=\lceil i/2\rceil}^{i-2} N^{-\kappa\sigma} j^{\kappa\sigma-2} \int_{x_j}^{x_{j+1}}(x_i-y)^{-1-\alpha}dy \\
		& \lesssim N^{-\kappa\sigma} i^{\kappa\sigma-2} \int_{x_{\lceil i/2 \rceil}}^{x_{i-1}} (x_i-y)^{-1-\alpha}dy, 
        \end{aligned}
        \end{equation*}
        where the equivalence $j^{\beta} \sim i^{\beta}$ for some $\beta=\kappa\sigma-2$ and $\lceil i/2 \rceil \le j \le i$ has been applied. Then, by directly evaluating the integral and applying \eqref{grad_mesh:e2} once more, we obtain
        \begin{equation}\label{err_hisL:e4b}
        \begin{aligned}
        I_{22}&\lesssim N^{-\kappa\sigma} i^{\kappa\sigma-2} [(x_i-x_{i-1})^{-\alpha}-(x_i-x_{\lceil i/2 \rceil})^{-\alpha}] \\
		& \lesssim N^{-\kappa\sigma} i^{\kappa\sigma-2} h_i^{-\alpha}
		 \lesssim N^{\kappa(\alpha-\sigma)} i^{-\kappa(\alpha-\sigma) -2+\alpha}.
	\end{aligned}
	\end{equation}

Finally, inserting the estimates \eqref{err_hisL:e4a}--\eqref{err_hisL:e4b} into \eqref{err_hisL:e4} yields the bound for $I_2$, and then inserting it with the bound $I_1$ in \eqref{err_hisL:e2} and \eqref{err_hisL:e3} into \eqref{err_hisL:e1}, observing the fact that $\kappa(1+\sigma) > 2$ implies $-\kappa(\alpha-\sigma)-2 >  -\kappa(1+\alpha)$, we arrive at the desired result. 
\end{proof}

Next, we give an estimate for the right-sided nonlocal integral error of the interpolation.
\begin{lemma}\label{lem:err_hisR}
	Assume that \eqref{model:regu} holds. Then, for $1 \leq i \leq N/2$, we have
	\begin{equation*}
		\Big| \int_{x_{i+1}}^{+\infty} \frac{\Pi_hu(y)-u(y)}{(y-x_i)^{1+\alpha}} dy \Big|
        \lesssim \left\{\begin{aligned}
		 & N^{\kappa(\alpha-\sigma)} i^{-\min\{\kappa(1+\alpha), \, \kappa(\alpha-\sigma)+2-\alpha\}}, & \kappa(1+\sigma)<2,\\
		 & N^{\kappa(\alpha-\sigma)} i^{-\kappa(1+\alpha)} \max\{\ln N,i^\alpha\}, & \kappa(1+\sigma)=2,\\
		 & N^{\kappa(\alpha-\sigma)} i^{-\kappa(\alpha-\sigma)-2+\alpha}, & \kappa(1+\sigma)>2.
		 \end{aligned}\right.
	\end{equation*}	
\end{lemma}
\begin{proof}	
Using once again the zero boundary condition in \eqref{model:nonlocal}, the integral can be decomposed as
	\begin{equation}\label{err_hisR:e1}
    \begin{aligned}
		\int_{x_{i+1}}^{+\infty} \frac{\Pi_hu(y)-u(y)}{(y-x_i)^{1+\alpha}} dy
        & = \sum_{j=i+1}^{N-2} \int_{x_{j}}^{x_{j+1}} \frac{\Pi_hu(y)-u(y)}{(y-x_i)^{1+\alpha}} dy + \int_{x_{N-1}}^{x_{N}} \frac{\Pi_hu(y)-u(y)}{(y-x_i)^{1+\alpha}} dy\\
        &=: I_3+I_4.
        \end{aligned}
	\end{equation}
	
	To estimate $I_3$, we first apply the standard linear interpolation error estimate and the regularity assumption \eqref{model:regu}, and then divide the summation into four parts as follows:
	\begin{equation}\label{err_hisR:e2}
    \begin{aligned}
		|I_3| & \lesssim \sum_{j=i+1}^{N-2} h_{j+1}^2 \max_{\xi\in[x_j,x_{j+1}]}|u''(\xi)| \int_{x_j}^{x_{j+1}}(y-x_i)^{-1-\alpha} dy \\
       & \lesssim \bigg[ \sum_{j=i+1}^{K} + \sum_{j=K+1}^{N/2} \bigg]  h_{j+1}^2 x_j^{\sigma-2}\int_{x_j}^{x_{j+1}}(y-x_i)^{-1-\alpha} dy \\
       &\quad + \bigg[ \sum_{j=N/2+1}^{\lceil\frac{3N}{4}\rceil -1} + \sum_{j=\lceil\frac{3N}{4}\rceil}^{N-2}\bigg] h_{j+1}^2 (2L-x_{j+1})^{\sigma-2} \int_{x_j}^{x_{j+1}}(y-x_i)^{-1-\alpha} dy \\
       & =: I_{31} + I_{32} + I_{33} + I_{34},
        \end{aligned}
	\end{equation}
	where $K = \min\{2i,N/2\}$, and $I_{32}$ is empty if $K = N/2$. 
    
    For $I_{31}$, analogous to the estimation of $I_{22}$, by using \eqref{grad_mesh:e1}--\eqref{grad_mesh:e2} and the fact that $j^{\beta} \sim i^{\beta}$ for some $\beta=\kappa\sigma-2$ and $i+1\leq j \leq K \le 2i$, we have
	\begin{equation*}
	\begin{aligned}
		I_{31} 
		 \lesssim  N^{-\kappa\sigma} i^{\kappa\sigma-2} \int_{x_{i+1}}^{x_{K+1}}(y-x_i)^{-1-\alpha} dy   \lesssim    N^{-\kappa\sigma} i^{\kappa\sigma-2} h_{i+1}^{-\alpha}
		 \lesssim N^{\kappa(\alpha-\sigma)} i^{-\kappa(\alpha-\sigma)-2+\alpha}.
	\end{aligned}
	\end{equation*}
    
	For $I_{32}$, we pay attention to the special case $K = 2i < N/2$. Similar to \eqref{equ:er_hisl:e1}, we have 
\begin{equation}\label{err_hisR:e3}
     \int_{x_j}^{x_{j+1}}(y-x_i)^{-1-\alpha} dy \le h_{j+1} (x_j-x_i)^{-1-\alpha}.
\end{equation}
We then apply \eqref{err_hisR:e3} to $I_{32}$ and use \eqref{grad_mesh:e1}--\eqref{grad_mesh:e2} to get 
	\begin{equation*}
		 I_{32}
         \lesssim \sum_{j=2i+1}^{N/2} h_{j+1}^3 x_j^{\sigma-2} (x_j-x_i)^{-1-\alpha}
         \lesssim N^{\kappa(\sigma-\alpha)} \sum_{j=2i+1}^{N/2} j^{\kappa(1+\sigma)-3}  \big(j^\kappa-i^\kappa\big)^{-1-\alpha} .
        \end{equation*}
   Moreover, by monotonicity, similar to \eqref{err_hisL:e4c}, we have
      $
    \big(j^\kappa-i^\kappa\big)^{-1-\alpha} 
    \lesssim j^{-\kappa(1+\alpha)}$ 
    for $j \ge 2i+1$. Hence, we further obtain
    \begin{equation*}
    \begin{aligned}
     I_{32} 
    & \lesssim N^{\kappa(\alpha-\sigma)} \sum_{j=2i+1}^{N/2} j^{-\kappa(\alpha-\sigma)-3} \lesssim N^{\kappa(\alpha-\sigma)}\int_{2i}^{+\infty} x^{-\kappa(\alpha-\sigma)-3} dx \\
     &\lesssim N^{\kappa(\alpha-\sigma)} i^{-\kappa(\alpha-\sigma)-2},
     \end{aligned}
    \end{equation*}
    where the regularity condition $\sigma\leq \frac{\alpha}{2}<\alpha$ has been used.
    
For $I_{33}$, 
using \eqref{grad_mesh:e1}--\eqref{grad_mesh:e2} and the fact that $(N-j)^{\kappa\sigma-2} \sim N^{\kappa\sigma-2}$ for $N/2+1 \leq j \leq K \le \lceil\frac{3N}{4}\rceil -1$, we have
	\begin{equation*}
	\begin{aligned}
		 I_{33} & \lesssim \sum_{j=N/2+1}^{\lceil\frac{3N}{4}\rceil-1} N^{-\kappa\sigma} (N-j)^{\kappa\sigma-2} \int_{x_j}^{x_{j+1}}(y-x_i)^{-1-\alpha}dy\\
		 &\lesssim N^{-2} \int_{x_{N/2+1}}^{x_{\lceil\frac{3N}{4}\rceil}} (y-x_i)^{-1-\alpha}dy \\
    & \lesssim N^{-2} (x_{N/2+1}-x_i)^{-\alpha}
		 \lesssim N^{-2+\alpha},
	\end{aligned}
	\end{equation*}
 where, by monotonicity, $(x_{N/2+1}-x_i)^{-\alpha} \le h_{N/2+1}^{-\alpha} \lesssim N^{\alpha}$ for $i\leq N/2$ has also been applied in the last step.
    
For $I_{34}$, following the same procedure as for $I_{21}$, by applying \eqref{err_hisR:e3} and noting that $(j^\kappa-i^\kappa)^{-1-\alpha} \lesssim [\lceil \frac{3N}{4} \rceil^\kappa - (\frac{N}{2})^{\kappa}]^{-1-\alpha} \lesssim N^{-\kappa(1+\alpha)}$ for $\lceil \frac{3N}{4}\rceil \leq j\leq N-2$ and $1\leq i\leq\frac{N}{2}$, we obtain
	\begin{equation*}
	\begin{aligned}
		I_{34}
		& \lesssim \sum_{j=\lceil\frac{3N}{4}\rceil}^{N-2} h_{j+1}^3 (2L-x_{j+1})^{\sigma-2}(x_j-x_i)^{-1-\alpha} \\
        & \lesssim N^{\kappa(\alpha-\sigma)} \sum_{j=\lceil\frac{3N}{4}\rceil}^{N-2} (N-j-1)^{\kappa(1+\sigma)-3}(j^\kappa-i^\kappa)^{-1-\alpha}\\
        &
        \lesssim N^{-\kappa(1+\sigma)} \sum_{j=1}^{\lceil\frac{N}{4}\rceil} j^{\kappa(1+\sigma)-3}
	 \lesssim \left\{
	 \begin{aligned}
	  & N^{-\kappa(1+\sigma)}, & \kappa(1+\sigma)<2, \\
	    & N^{-2} \ln N, & \kappa(1+\sigma)=2, \\ 
	    & N^{-2}, & \kappa(1+\sigma)>2 .
	 \end{aligned}\right.
     \end{aligned}
	\end{equation*} 
    
Now, inserting the estimates $I_{31}$--$ I_{34}$ into \eqref{err_hisR:e2} and noting that
\begin{equation}\label{err_hisR:e4}
   \begin{aligned}
   N^{-\kappa(1+\sigma)} & = N^{\kappa(\alpha-\sigma)} N^{-\kappa(1+\alpha)} \leq N^{\kappa(\alpha-\sigma)} i^{-\kappa(1+\alpha)},\\
  N^{-2+\alpha} & = N^{\kappa(\alpha-\sigma)}N^{-\kappa(\alpha-\sigma)-2+\alpha} \leq  N^{\kappa(\alpha-\sigma)}i^{-\kappa(\alpha-\sigma)-2+\alpha},
   \end{aligned}
\end{equation} 
we get the estimate for $I_3$ that
	\begin{equation} \label{err_hisR:e5}
    \begin{aligned}
		|I_3| & \lesssim \left\{\begin{aligned}
		 & N^{\kappa(\alpha-\sigma)} i^{-\min\{\kappa(1+\alpha), \, \kappa(\alpha-\sigma)+2-\alpha\}}, & \kappa(1+\sigma)<2,\\
		 & N^{\kappa(\alpha-\sigma)} i^{-\kappa(1+\alpha)} \max\{\ln N,i^\alpha\}, & \kappa(1+\sigma)=2,\\
		 & N^{\kappa(\alpha-\sigma)} i^{-\kappa(\alpha-\sigma)-2+\alpha}, & \kappa(1+\sigma)>2.
		 \end{aligned}\right.
        \end{aligned}
	\end{equation}

Analogous to the estimate of $I_{1}$, we use the definition of $\Pi_h u(y)$, the regularity assumption \eqref{model:regu} and the estimate \eqref{err_hisR:e3} with $j=N-1$ to estimate $I_4$ that
	\begin{equation*} 
	\begin{aligned}
		|I_{4}| 
        &= \Bigg| \int_{x_{N-1}}^{x_N} (y-x_i)^{-1-\alpha} \Big[\frac{y-x_{N-1}}{h_N} \int_{y}^{x_N} u'(s)ds  -  \frac{x_N-y}{h_N} \int_{x_{N-1}}^y  u'(s)ds \Big]dy\Bigg|\\
		& \leq \int_{x_{N-1}}^{x_N} (y-x_i)^{-1-\alpha} dy\int_{x_{N-1}}^{x_N} |u'(s)| ds \lesssim \int_{x_{N-1}}^{x_N} (y-x_i)^{-1-\alpha} dy\int_{x_{N-1}}^{x_N} (2L-s)^{\sigma-1} ds\\
        & \leq h_N^{1+\sigma} (x_{N-1}-x_i)^{-1-\alpha}.
        \end{aligned}
        \end{equation*}
     Since $x_{N-1}-x_i \ge x_{N-1}-x_{N/2} 
     = L(1-(2/N)^\kappa) \gtrsim 1$ for $1\leq i\leq N/2$ and $N \ge 3$ large enough, it follows that $(x_{N-1}-x_i)^{-1-\alpha}\lesssim 1$. Thus, by \eqref{grad_mesh:e2} and \eqref{err_hisR:e4} we obtain
        \begin{equation}\label{err_hisR:e6}
		|I_4| \lesssim N^{-\kappa(1+\sigma)} \lesssim N^{\kappa(\alpha-\sigma)} i^{-\kappa(1+\alpha)}.
	\end{equation}
	
Therefore, by collecting the bounds in \eqref{err_hisR:e5} for $I_{3}$ and \eqref{err_hisR:e6} for $I_4$ into \eqref{err_hisR:e1}, the proof is complete.
\end{proof}

Finally, we give an estimate for the local integral error of the interpolation.
\begin{lemma}\label{lem:err_loc}
	Assume that \eqref{model:regu} holds. Then, for $1 \leq i \leq N/2$, we have
	\begin{equation*}
		\Big| \int_{x_{i-1}}^{x_{i+1}} \frac{\Pi_hu(y)-u(y)}{|x_i-y|^{1+\alpha}} dy\Big|  \lesssim   N^{\kappa(\alpha-\sigma)} i^{-\kappa(\alpha-\sigma)-2+\alpha}.
	\end{equation*}	
\end{lemma}
\begin{proof} We begin by decomposing the local integral into two parts and performing integration by parts on each subinterval. By applying the standard interpolation error estimate to the resulting terms, we arrive at 
\begin{equation}\label{err_loc:e1}
\begin{aligned}
& \quad \Big| \int_{x_{i-1}}^{x_{i+1}} \frac{\Pi_hu(y)-u(y)}{|x_i-y|^{1+\alpha}} dy\Big|   \leq   \Big| \int_{x_{i-1}}^{x_{i}} \frac{\Pi_hu(y)-u(y)}{(x_i-y)^{1+\alpha}} dy\Big| + \Big| \int_{x_{i}}^{x_{i+1}} \frac{\Pi_hu(y)-u(y)}{(y-x_i)^{1+\alpha}} dy\Big|\\
& = \frac{1}{\alpha} \Big| \int_{x_{i-1}}^{x_{i}} (\Pi_hu(y)-u(y))' (x_i-y)^{-\alpha} dy\Big|  +  \frac{1}{\alpha} \Big| \int_{x_{i}}^{x_{i+1}} (\Pi_hu(y)-u(y))' (x_i-y)^{-\alpha} dy\Big|\\
& \lesssim h_i \max_{\xi\in[x_{i-1},x_i]}|u''(\xi)| \int_{x_{i-1}}^{x_{i}} (x_i-y)^{-\alpha}dy  + h_{i+1} \max_{\xi\in[x_i,x_{i+1}]}|u''(\xi)| \int_{x_i}^{x_{i+1}} (y-x_i)^{-\alpha} dy\\
& = \frac{h_i^{2-\alpha}}{1-\alpha}\max_{\xi\in[x_{i-1},x_i]}|u''(\xi)| + \frac{h_{i+1}^{2-\alpha}}{1-\alpha}\max_{\xi\in[x_i,x_{i+1}]}|u''(\xi)|.
\end{aligned}
\end{equation}

Using the regularity assumption \eqref{model:regu} and the mesh relations \eqref{grad_mesh:e1}--\eqref{grad_mesh:e2}, for $1 \leq i \leq N/2-1$, we obtain from \eqref{err_loc:e1} that
 \begin{equation*}
  \Big| \int_{x_{i-1}}^{x_{i+1}} \frac{\Pi_hu(y)-u(y)}{|x_i-y|^{1+\alpha}} dy\Big|    \lesssim h_i^{2-\alpha} x_{i-1}^{\sigma-2} + h_{i+1}^{2-\alpha} x_{i}^{\sigma-2}   \lesssim N^{\kappa(\alpha-\sigma)} i^{-\kappa(\alpha-\sigma)-2+\alpha},
\end{equation*}
which proves the conclusion.
In particular, for $ i= N/2$, now the estimate for the last term in \eqref{err_loc:e1} is changed to $\max_{\xi\in[x_i,x_{i+1}]}|u''(\xi)| \lesssim (2L-x_{N/2+1})^{\sigma-2}$, and by symmetric, $2L-x_{N/2+1}=x_{N/2-1}$. Thus, we have
 \begin{equation*}
\begin{aligned}
 & \quad \Big| \int_{x_{N/2-1}}^{x_{N/2+1}} \frac{\Pi_hu(y)-u(y)}{|x_{N/2}-y|^{1+\alpha}} dy\Big|   
 \lesssim h_{N/2}^{2-\alpha} x_{N/2-1}^{\sigma-2} + h_{N/2+1}^{2-\alpha} x_{N/2-1}^{\sigma-2} 
 \lesssim N^{-2+\alpha},
\end{aligned}
\end{equation*}
which combines with \eqref{err_hisR:e4} also proves the conclusion.
\end{proof}

Based on Lemmas~\ref{lem:err_hisL}--\ref{lem:err_loc}, we now establish the following truncation error estimate
for the fast collocation scheme \eqref{sch:fast_Collocation}.
\begin{lemma}\label{lem:loc-tru-err} 	Assume that \eqref{model:regu} holds. Then the truncation error of the fast collocation scheme \eqref{sch:fast_Collocation} satisfies
	\begin{equation*}
    \begin{aligned}
		|R_i| & = |(-\Delta)^{\frac{\alpha}{2}}u(x_i) - \mL_h[u]_i|\\
        & \lesssim \left\{\begin{aligned}
		 & N^{\kappa(\alpha-\sigma)} i^{-\min\{\kappa(1+\alpha), \, \kappa(\alpha-\sigma)+2-\alpha\}} + \epsilon, & \kappa(1+\sigma)<2,\\
		 & N^{\kappa(\alpha-\sigma)} i^{-\kappa(1+\alpha)} \max\{\ln N,i^\alpha\} + \epsilon, & \kappa(1+\sigma)=2,\\
		 & N^{\kappa(\alpha-\sigma)} i^{-\kappa(\alpha-\sigma)-2+\alpha} + \epsilon, & \kappa(1+\sigma)>2,
		 \end{aligned}\right.
        \end{aligned}
	\end{equation*}
    for $1 \leq i \leq N/2$, and similarly
    \begin{equation*}
		|R_i|  \lesssim \left\{\begin{aligned}
		 & N^{\kappa(\alpha-\sigma)} (N-i)^{-\min\{\kappa(1+\alpha), \, \kappa(\alpha-\sigma)+2-\alpha\}} + \epsilon, & \kappa(1+\sigma)<2,\\
		 & N^{\kappa(\alpha-\sigma)} (N-i)^{-\kappa(1+\alpha)} \max\{\ln N,(N-i)^\alpha\} + \epsilon, & \kappa(1+\sigma)=2,\\
		 & N^{\kappa(\alpha-\sigma)} (N-i)^{-\kappa(\alpha-\sigma)-2+\alpha} + \epsilon, & \kappa(1+\sigma)>2,
		 \end{aligned}\right.
	\end{equation*}
    for $N/2+1 \leq i\leq N-1$.
\end{lemma}
\begin{proof} We split the local truncation error into two parts
	\begin{equation}\label{equ:loc_truc1}
		\begin{aligned}
		|R_i|&\leq |(-\Delta)^{\frac{\alpha}{2}}u(x_i)-\mL[u]_i|+|\mL[u]_i-\mL_hu(x_i)|=: I_5+ I_6,
		\end{aligned}
	\end{equation}
    i.e., error of direct collocation approximation to the model and error of the fast collocation scheme \eqref{sch:fast_Collocation} to the direct collocation scheme \eqref{sch:direct_Collocation}.
    
For the first term, Lemmas \ref{lem:err_hisL}--\ref{lem:err_loc} imply the direct collocation approximation error
    \begin{equation}\label{equ:loc1}
    |I_5| 
     \lesssim \left\{\begin{aligned}
		 & N^{\kappa(\alpha-\sigma)} i^{-\min\{\kappa(1+\alpha), \, \kappa(\alpha-\sigma)+2-\alpha\}}, & \kappa(1+\sigma)<2,\\
		 & N^{\kappa(\alpha-\sigma)} i^{-\kappa(1+\alpha)} \max\{\ln N,i^\alpha\}, & \kappa(1+\sigma)=2,\\
		 & N^{\kappa(\alpha-\sigma)} i^{-\kappa(\alpha-\sigma)-2+\alpha}, & \kappa(1+\sigma)>2,
		 \end{aligned}\right.
    \end{equation}
    for $1\leq i\leq N/2$.
Since the symmetric graded grid \eqref{grad_mesh:e1} is considered, a truncation error estimate for $N/2+1 \leq i \leq N-1$ can be similarly proved, that is,
    \begin{equation}\label{equ:loc2}
    |I_5|
    \lesssim \left\{\begin{aligned}
		 & N^{\kappa(\alpha-\sigma)} (N-i)^{-\min\{\kappa(1+\alpha), \, \kappa(\alpha-\sigma)+2-\alpha\}}, & \kappa(1+\sigma)<2,\\
		 & N^{\kappa(\alpha-\sigma)} (N-i)^{-\kappa(1+\alpha)} \max\{\ln N,(N-i)^\alpha\}, & \kappa(1+\sigma)=2,\\
		 & N^{\kappa(\alpha-\sigma)} (N-i)^{-\kappa(\alpha-\sigma)-2+\alpha}, & \kappa(1+\sigma)>2.
		 \end{aligned}\right.
    \end{equation}
	
For the second term, as both schemes have the same local integration approximation, it follows from Lemma \ref{lem:soe} that
	\begin{equation}\label{equ:loc3}
		\begin{aligned}
		|I_6|
        & \leq \Bigg| \int_{-\infty}^{x_{i-1}} \Pi_hu(y) \Big((x_i-y)^{-1-\alpha} - \sum_{s=1}^{N_{e}} \theta_s e^{-\lambda_s(x_i-y)} \Big)dy\\
		& \quad +\int_{x_{i+1}}^{+\infty} \Pi_hu(y) \Big((y-x_i)^{-1-\alpha} -\sum_{s=1}^{N_{e}} \theta_s e^{-\lambda_s(y-x_i)} \Big)dy\Bigg|\\
		& \leq \epsilon \Bigg[\int_{-\infty}^{x_{i-1}}|\Pi_hu(y)|dy + \int_{x_{i+1}}^{+\infty}|\Pi_hu(y)|dy\Bigg]
        \le \epsilon \max_{\xi\in[a, b]}|u(\xi)|.
		\end{aligned}
	\end{equation}
	
Thus, substituting the bounds in \eqref{equ:loc1}--\eqref{equ:loc3} into \eqref{equ:loc_truc1} gives the claimed truncation error estimate.
\end{proof}
\begin{theorem}\label{thm:conv}
Suppose that $u(x)\in C^2(0, 2L)$ and assumption \eqref{model:regu} holds. Let $\{U_i\}$ denote the numerical solution of the fast collocation scheme \eqref{sch:fast_Collocation}. Then, for sufficiently small $\epsilon$, we have
\begin{equation*}
\max_{1\leq i\leq N-1} |u(x_i)-U_i| \lesssim \left\{\begin{aligned}
& N^{-\kappa\sigma} + \epsilon, & \kappa(1+\sigma) < 2,\\
& N^{-\kappa \sigma}\ln N + \epsilon, & \kappa(1+\sigma) = 2,\\
& N^{-\min\{\kappa\sigma,2-\alpha\}}+ \epsilon,  & \kappa(1+\sigma) > 2.
\end{aligned}\right.
\end{equation*}
\end{theorem}
\begin{proof}
For $i=0,1,\dots,N$, set $e_i:=u(x_i)-U_i$ with $e_0=e_N=0$. Then the error equation reads
\begin{equation*}
\sum_{j=1}^{N-1} a_{i, j} e_j = R_i, \quad i=1,2,\dots, N-1.
\end{equation*}

Let $i_0$ be the index such that $|e_{i_0}| = \max_{1\leq i\leq N-1} |e_i|$. Recall from Theorem \ref{thm:dia-dom-mat} that $a_{i_0,i_0}>0$ and $a_{i_0,j}<0$ for all $j \neq i_0$. Hence, we have
\begin{equation*}
\begin{aligned}
|R_{i_0}| & = \Bigg| a_{i_0,i_0} e_{i_0} + \sum_{j=1,j\neq i_0}^{N-1} a_{i_0,j} e_j \Bigg|  \geq  a_{i_0,i_0} |e_{i_0}| - \sum_{j=1,j\neq i_0}^{N-1} |a_{i_0,j}| |e_j|\\
&\geq \left( a_{i_0,i_0}   + \sum_{j=1,j\neq i_0}^{N-1} a_{i_0,j} \right) |e_{i_0}| =  |e_{i_0}|\sum_{j=1}^{N-1} a_{i_0,j},
\end{aligned}
\end{equation*}
which, by $\eqref{matrix-form:a5}$, further implies that
\begin{equation}\label{equ:cov2}
 \max_{1\leq i\leq N-1} |e_i| = |e_{i_0}| \leq R_{i_0}/\sum_{j=1}^{N-1} a_{i_0,j} \leq R_{i_0} / \big( \Xi_{i_0} -2L\epsilon \big).
\end{equation}

Without loss of generality, we assume that $1 \leq i_0 \le N/2$, and the discussion for the case $i_0 > N/2$ is analogous. By the mean value theorem and \eqref{grad_mesh:e1}, we have
\begin{equation*}
\Xi_{i_0}-2L \epsilon   \geq  \frac{x_{i_0}^{-\alpha}+(2L-x_{i_0})^{-\alpha}}{\alpha} - 2L \epsilon 
      \ge \frac{L^{-\alpha}}{\alpha \cdot 2^{\kappa \alpha}}N^{\kappa \alpha} i_0^{-\kappa \alpha}-2L\epsilon \ge \frac{L^{-\alpha}}{\alpha \cdot 2^{\kappa \alpha+1}} N^{\kappa \alpha} i_0^{-\kappa \alpha},
\end{equation*}
if we choose $\epsilon \leq \frac{1}{4\alpha L^{1+\alpha}} \leq \frac{1}{4\alpha  L^{1+\alpha}2^{\kappa\alpha}}  N^{\kappa\alpha}i_0^{-\kappa\alpha}$  small enough.
Thus, substituting this estimate into \eqref{equ:cov2} and applying Lemma \ref{lem:loc-tru-err}, we obtain
\begin{equation}\label{equ:cov1}
\begin{aligned}
\max_{1\leq i\leq N-1} |e_i| & \lesssim \left\{\begin{aligned}
		 & N^{-\kappa\sigma} i_0^{-\min\{\kappa, \, -\kappa\sigma+2-\alpha\}} + \epsilon, & \kappa(1+\sigma)<2,\\
		 & N^{-\kappa\sigma} i_0^{-\kappa} \max\{\ln N,i_0^\alpha\} + \epsilon, & \kappa(1+\sigma)=2,\\
		 & N^{-\kappa\sigma} i_0^{\kappa\sigma-2+\alpha} + \epsilon, & \kappa(1+\sigma)>2.
		 \end{aligned}\right.
\end{aligned}
\end{equation}

We next give a detailed discussion on \eqref{equ:cov1}. 
\begin{itemize}
\item Case i. $\kappa(1+\sigma)<2$. Note that $\min \{\kappa, -\kappa\sigma+2-\alpha\}\ge  \min \{\kappa, \kappa-\alpha\} >0$ as $\alpha\in(0,1)$ and $\kappa \geq 1$. Then, from \eqref{equ:cov1} we have
\begin{equation*}
\max_{1\leq i\leq N-1} |e_i| \lesssim N^{-\kappa\sigma} + \epsilon.
\end{equation*}

\item Case ii. $\kappa(1+\sigma)=2$. As $\kappa>\alpha>0$, one has $i_0^{-\kappa}\le 1$ and $i_0^{\alpha-\kappa}\le 1$. Then, from \eqref{equ:cov1} we have
\begin{equation*}
\max_{1\leq i\leq N-1} |e_i| \lesssim N^{-\kappa\sigma}\ln N + \epsilon.
\end{equation*}

\item Case iii.  $\kappa(1+\sigma)>2$. From \eqref{equ:cov1}, we have
\begin{equation*}
\begin{aligned}
\max_{1\leq i\leq N-1} |e_i| & \lesssim   \left\{\begin{aligned}
&N^{-\kappa\sigma} i_0^{\kappa\sigma-2+\alpha}+ \epsilon \lesssim N^{-\kappa\sigma} + \epsilon, & \kappa\sigma < 2-\alpha,\\
&N^{-\kappa\sigma}N^{\kappa\sigma-2+\alpha}+ \epsilon =N^{-(2-\alpha)}+ \epsilon ,  & \kappa\sigma \geq 2-\alpha.
\end{aligned}\right.
\end{aligned}
\end{equation*}
\end{itemize}
This completes the proof. 
\end{proof}

\begin{remark}
When $\kappa(1+\sigma)>2$ and simultaneously $\kappa\sigma \ge 2-\alpha$, that is,
$
\kappa \ge \frac{2-\alpha}{\sigma},
$
the error estimate in Theorem \ref{thm:conv} attains the optimal convergence order $2-\alpha$, i.e.,
\begin{equation*}
\max_{1\leq i \leq N-1}|u(x_i)-u_i| \lesssim N^{-(2-\alpha)}+\epsilon.
\end{equation*}
\end{remark}

For the modified fast collocation scheme \eqref{sch:fast_collocation_md}, the results of Lemmas~\ref{lem:err_hisL}--\ref{lem:err_hisR} remain valid. However, the local error bound given in Lemma~\ref{lem:err_loc} is modified as follows.
\begin{lemma}\label{lem:err_loc_md}
	Assume that \eqref{model:regu} holds. Then, for $1 \leq i \leq N/2$, we have
	\begin{equation*}
		\Big| \int_{x_{i-1}}^{x_{i+1}} \frac{u(x_i)+u'(x_i)(y-x_i)-u(y)}{|x_i-y|^{1+\alpha}} dy\Big|  \lesssim   N^{\kappa(\alpha-\sigma)} i^{-\kappa(\alpha-\sigma)-2+\alpha}.
	\end{equation*}	
\end{lemma}
\begin{proof}
By decomposing the integration into two parts and applying the Taylor expansion, we have
\begin{equation*}
\begin{aligned}
& \quad \Big| \int_{x_{i-1}}^{x_{i+1}} \frac{u(x_i)+u'(x_i)(y-x_i)-u(y)}{|x_i-y|^{1+\alpha}} dy\Big| \\
&\lesssim \max_{\xi\in[x_{i-1},x_i]}|u''(\xi)| \int_{x_{i-1}}^{x_{i}} (x_i-y)^{1-\alpha}dy  +  \max_{\xi\in[x_i,x_{i+1}]}|u''(\xi)| \int_{x_i}^{x_{i+1}} (y-x_i)^{1-\alpha} dy\\
& = \frac{h_i^{2-\alpha}}{2-\alpha}\max_{\xi\in[x_{i-1},x_i]}|u''(\xi)| + \frac{h_{i+1}^{2-\alpha}}{2-\alpha}\max_{\xi\in[x_i,x_{i+1}]}|u''(\xi)|.
\end{aligned}
\end{equation*}
The remaining analysis follows exactly the same argument as in the proof of Lemma \ref{lem:err_loc}. Hence, we have
\begin{equation*}
		\Big| \int_{x_{i-1}}^{x_{i+1}} \frac{u(x_i)+u'(x_i)(y-x_i)-u(y)}{|x_i-y|^{1+\alpha}} dy\Big|  \lesssim   N^{\kappa(\alpha-\sigma)} i^{-\kappa(\alpha-\sigma)-2+\alpha},
	\end{equation*}	
which proves the lemma.
\end{proof}

Thus, by combining Lemmas~\ref{lem:err_hisL}--\ref{lem:err_hisR} and \ref{lem:err_loc_md}, and similar to Lemma \ref{lem:loc-tru-err}, one can derive the truncation error estimate for the modified fast collocation scheme \eqref{sch:fast_collocation_md}.
\begin{lemma}\label{lem:loc-tru-err-md} Assume that \eqref{model:regu} holds. Then the truncation error of the modified fast collocation scheme \eqref{sch:fast_collocation_md} satisfies
	\begin{equation*}
    \begin{aligned}
		|R_i| & = |(-\Delta)^{\frac{\alpha}{2}}u(x_i) - \mL_h^{(m)}[u]_i|\\
        & \lesssim \left\{\begin{aligned}
		 & N^{\kappa(\alpha-\sigma)} i^{-\min\{\kappa(1+\alpha), \, \kappa(\alpha-\sigma)+2-\alpha\}} + \epsilon, & \kappa(1+\sigma)<2,\\
		 & N^{\kappa(\alpha-\sigma)} i^{-\kappa(1+\alpha)} \max\{\ln N,i^\alpha\} + \epsilon, & \kappa(1+\sigma)=2,\\
		 & N^{\kappa(\alpha-\sigma)} i^{-\kappa(\alpha-\sigma)-2+\alpha} + \epsilon, & \kappa(1+\sigma)>2.
		 \end{aligned}\right.
        \end{aligned}
	\end{equation*}
    for $1 \leq i \leq N/2$, and similarly
    \begin{equation*}
		|R_i|  \lesssim \left\{\begin{aligned}
		 & N^{\kappa(\alpha-\sigma)} (N-i)^{-\min\{\kappa(1+\alpha), \, \kappa(\alpha-\sigma)+2-\alpha\}} + \epsilon, & \kappa(1+\sigma)<2,\\
		 & N^{\kappa(\alpha-\sigma)} (N-i)^{-\kappa(1+\alpha)} \max\{\ln N,(N-i)^\alpha\} + \epsilon, & \kappa(1+\sigma)=2,\\
		 & N^{\kappa(\alpha-\sigma)} (N-i)^{-\kappa(\alpha-\sigma)-2+\alpha} + \epsilon, & \kappa(1+\sigma)>2.
		 \end{aligned}\right.
	\end{equation*}
    for $N/2+1 \leq i\leq N-1$.
\end{lemma}

Moreover, similarly to the proof of Theorem \ref{thm:conv}, the main error estimate for the modified fast collocation scheme \eqref{sch:fast_collocation_md} is presented as follows.
\begin{theorem}\label{thm:conv-md}
Suppose that $u(x)\in C^2(0, 2L)$ and assumption \eqref{model:regu} holds. Let $\{U_i\}$ denote the numerical solution of the modified fast collocation scheme \eqref{sch:fast_collocation_md}. Then, for sufficiently small $\epsilon$, we have
\begin{equation*}
\max_{1\leq i\leq N-1} |u(x_i)-U_i| \lesssim \left\{
 \begin{aligned}
& N^{-\kappa\sigma} + \epsilon, & \kappa(1+\sigma) < 2,\\
& N^{-\kappa \sigma}\ln N + \epsilon, & \kappa(1+\sigma) = 2,\\
& N^{-\min\{\kappa\sigma,2-\alpha\}}+ \epsilon ,  & \kappa(1+\sigma) > 2.
 \end{aligned}\right.
\end{equation*}
\end{theorem}

\section{Numerical experiments}\label{sec:num}
In this section, we carry out numerical experiments to demonstrate the performance of the newly developed fast collocation scheme \eqref{sch:fast_Collocation} and the modified scheme \eqref{sch:fast_collocation_md}. The simulations will be terminated when either of the following conditions is met: the relative residual error falls below $10^{-8}$, or the number of iterations exceeds the maximum allowed. The tolerance in the SOE approximation is set to $\epsilon=10^{-8}$. All numerical experiments are performed using Matlab R2019b on a laptop with the configuration: AMD Ryzen 7 5800H with Radeon Graphics @ 3.20 GHz and 16.00 GB RAM. Besides, the errors and corresponding convergence orders (Cov.) are measured by  
$$
\|e_N\|_{\infty}:=\max_{1\leq i\leq N-1}|u(x_i)-U_i|, \quad \text{Cov.}=\log_2\left(\|e_N\|_{\infty}/ \|e_{2N}\|_\infty \right).
$$

One of the few examples where the analytical solution to the fractional Laplacian Dirichlet problem 
\begin{equation*}
\left\{
\begin{aligned}
 (-\Delta)^{\frac{\alpha}{2}}u  & = 1,  &&\text{in}\ (0,2),\\
 u  & = 0,  &&\text{in}\, (-\infty,0]\cup[2,+\infty),
\end{aligned}\right.
\end{equation*}
is given by \cite{G61} with $\sigma=\frac{\alpha}{2}$ in \eqref{model:regu} that
\begin{equation*}
u(x)=\frac{2^{-\alpha}\sqrt{\pi}}{\Gamma (\frac{1+\alpha}{2} ) \Gamma (1+\frac{\alpha}{2} )} \left[x(2-x)\right]^{\frac{\alpha}{2}}.
\end{equation*}

First, we test the accuracy of the proposed fast collocation scheme \eqref{sch:fast_Collocation} and the modified scheme \eqref{sch:fast_collocation_md}. Numerical results for various combinations of fractional order $\alpha$ $(0<\alpha<1)$ and the graded parameter $\kappa$ are presented in Tables \ref{tab:cov1}--\ref{tab:cov3}. 
It clearly shows that both schemes exhibit the same convergence rate. As observed, the convergence order is almost $\mO(N^{-\kappa\sigma})$ when $\kappa(1+\sigma)\leq 2$ and $\mO(N^{-\min\{2-\alpha,\kappa\sigma\}})$ when $\kappa(1+\sigma)>2$, which aligns closely with the theoretical result established in Theorems \ref{thm:conv} and \ref{thm:conv-md} for small $\epsilon$. In particular, the optimal convergence order $\mO(N^{-(2-\alpha)})$ can be achieved when $\kappa={(2-\alpha)}/{\sigma}$. Moreover, the modified fast collocation scheme achieves smaller errors on uniform grids, whereas the original fast collocation scheme performs better on nonuniform grids with $\kappa = (2-\alpha)/\sigma$, especially for small $\alpha$. These results suggest that, for $\alpha \in (0,1)$, the original scheme is more suitable for problems on nonuniform grids, while the modified fast collocation scheme is preferable for uniform grids. 
\begin{table}[!thbp]\small
\vspace{-12pt}
\renewcommand{\arraystretch}{0.75}
	\centering 
	\caption{Errors and convergence orders of two fast schemes for $\alpha=0.8$}\label{tab:cov1}
	\setlength{\tabcolsep}{3mm}
	\begin{tabular}{ c c c c c c c c c c}
		\toprule
       \multirow{2}{*}{$\kappa$} &\multirow{2}{*}{$N$} & \multicolumn{2}{c}{the original fast scheme \eqref{sch:fast_Collocation}} & \multicolumn{2}{c}{the modified scheme \eqref{sch:fast_collocation_md}} \\ 
		\cmidrule(r){3-4}     \cmidrule(r){5-6}    	
		  &  & $\|e_N\|_{\infty}$ & Cov. & $\|e_N\|_{\infty}$ & Cov.  \\
        \midrule
		  & $2^6$    & 1.5655e-01 & --- & 4.1364e-02 & --- \\
$1$ & $2^7$ & 1.1754e-01 & 0.4136 & 3.1191e-02 & 0.4073 \\
		& $2^8$    & 8.8630e-02 & 0.4072 & 2.3570e-02 & 0.4041 \\
& $2^9$    & 6.6991e-02 & 0.4038 & 1.7835e-02 & 0.4023 \\[0.03in]
        & &\multicolumn{2}{c}{$\kappa \sigma =0.40$ } & \multicolumn{2}{c}{$\kappa \sigma =0.40$ } \\
        \midrule
		  & $2^6$    & 1.0577e-01 & --- & 6.0620e-02 & --- \\
	$2/(1+\sigma)$ & $2^7$ & 7.0500e-02 & 0.5853 & 4.0661e-02 & 0.5761 \\
		  & $2^8$    & 4.7223e-02 & 0.5781 & 2.7320e-02 & 0.5737 \\
		 & $2^9$    & 3.1709e-02 & 0.5746 & 1.8371e-02 & 0.5725 \\[0.03in]
		& &	\multicolumn{2}{c}{$ \kappa \sigma \approx 0.57$ }   
		   & \multicolumn{2}{c}{$  \kappa \sigma \approx 0.57$ } \\
           \midrule
		  & $2^6$    & 1.0088e-01 & --- & 6.3670e-02 & --- \\
$(2-\alpha)/2\sigma$ & $2^7$ & 6.5980e-02 & 0.6126 & 4.1885e-02 & 0.6042 \\
	& $2^8$    & 4.3350e-02 & 0.6060 & 2.7596e-02 & 0.6020 \\
		& $2^9$    & 2.8545e-02 & 0.6028 & 1.8195e-02 & 0.6009 \\[0.03in]
        & &\multicolumn{2}{l}{$\min\{ 2-\alpha, \kappa \sigma \}=\kappa \sigma =0.60$ } & \multicolumn{2}{l}{$\min\{ 2-\alpha, \kappa \sigma \} =\kappa \sigma =0.60$ } \\
        \midrule
		  & $2^6$    & 7.0863e-02 & --- & 1.0682e-01 & --- \\
	$(2-\alpha)/\sigma$ & $2^7$ & 4.6506e-02 & 1.1679 & 4.6506e-02 & 1.1997 \\
		  & $2^8$    & 1.3861e-02 & 1.1861 & 2.0245e-02 & 1.1999 \\
		& $2^9$    & 6.0603e-03 & 1.1936 & 8.8126e-03 & 1.1999 \\[0.03in]
		& &	\multicolumn{2}{c}{$ \min\{ 2-\alpha, \kappa \sigma \} =2-\alpha =1.20$ }   
		   & \multicolumn{2}{c}{$ \min\{ 2-\alpha, \kappa \sigma \} =2-\alpha =1.20$ } \\
		\bottomrule
	\end{tabular}
\end{table}
\begin{table}[!htbp]\small
\vspace{-12pt}
\renewcommand{\arraystretch}{0.75}
	\centering 
	\caption{Errors and convergence orders of two fast schemes for $\alpha=0.6$}
	\setlength{\tabcolsep}{3mm}
	\begin{tabular}{ c c c c c c c c c c}
		\toprule
       \multirow{2}{*}{$\kappa$} &\multirow{2}{*}{$N$} & \multicolumn{2}{c}{the original fast scheme \eqref{sch:fast_Collocation}} & \multicolumn{2}{c}{the modified scheme \eqref{sch:fast_collocation_md}} \\ 
		\cmidrule(r){3-4}     \cmidrule(r){5-6}    	
		  &  & $\|e_N\|_{\infty}$ & Cov. & $\|e_N\|_{\infty}$ & Cov.  \\
        \midrule
		  & $2^6$    & 1.3947e-01 & --- & 2.9979e-02 & --- \\
$1$     & $2^7$    & 1.1277e-01 & 0.3065 & 2.2644e-02 & 0.2958 \\
		& $2^8$    & 9.1393e-02 & 0.3033 & 1.8419e-02 & 0.2979 \\
        & $2^9$    & 7.4150e-02 & 0.3016 & 1.4971e-02 & 0.2990 \\[0.03in]
        & &\multicolumn{2}{c}{$\kappa \sigma =0.30$ } & \multicolumn{2}{c}{$\kappa \sigma =0.30$ } \\
        \midrule
		  & $2^6$    & 9.6555e-02 & --- & 6.8511e-02 & --- \\
$2/(1+\sigma)$ & $2^7$ & 6.9978e-02 & 0.4644 & 4.9703e-02 & 0.4630 \\
		  & $2^8$    & 5.0776e-02 & 0.4628 & 3.6079e-02 & 0.4622 \\
		  & $2^9$    & 3.6862e-02 & 0.4620 & 2.6196e-02 & 0.4618 \\[0.03in]
		& &	\multicolumn{2}{c}{$ \kappa \sigma \approx 0.46$ }   
		   & \multicolumn{2}{c}{$  \kappa \sigma \approx 0.46$ } \\
           \midrule
		  & $2^6$    & 5.4187e-02 & --- & 1.7638e-01 & --- \\
$(2-\alpha)/2\sigma$ & $2^7$ & 3.3337e-02 & 0.7008 & 1.0855e-01 & 0.7003 \\
	    & $2^8$    & 2.0517e-02 & 0.7003 & 6.6817e-02 & 0.7001 \\
		& $2^9$    & 1.2628e-02 & 0.7001 & 4.1129e-02 & 0.7000 \\[0.03in]
        & &\multicolumn{2}{l}{$\min\{ 2-\alpha, \kappa \sigma \} =\kappa \sigma=0.70$ } & \multicolumn{2}{l}{$\min\{ 2-\alpha, \kappa \sigma \} =\kappa \sigma=0.70$ } \\
        \midrule
		  & $2^6$    & 2.1898e-02 & --- & 3.3124e-01 & --- \\
	$(2-\alpha)/\sigma$ & $2^7$ & 8.5536e-03 & 1.3562 & 1.2559e-01 & 1.3991 \\
		  & $2^8$    & 3.3051e-03 & 1.3718 & 4.7602e-02 & 1.3997 \\
		& $2^9$    & 1.2687e-03 & 1.3813 & 1.8039e-02 & 1.3999 \\[0.03in]
		& &	\multicolumn{2}{l}{$ \min\{ 2-\alpha, \kappa \sigma \} =2-\alpha =1.40$ }   
		   & \multicolumn{2}{l}{$ \min\{ 2-\alpha, \kappa \sigma \} =2-\alpha =1.40$ } \\
		\bottomrule
	\end{tabular}
\end{table}	

\begin{table}[!htbp]\small
\vspace{-12pt}
\renewcommand{\arraystretch}{0.75}
	\centering 
	\caption{Errors and convergence orders of two fast schemes for $\alpha=0.4$}\label{tab:cov3}
	\setlength{\tabcolsep}{3mm}
	\begin{tabular}{ c c c c c c c c c c}
		\toprule
       \multirow{2}{*}{$\kappa$} &\multirow{2}{*}{$N$} & \multicolumn{2}{c}{the original fast scheme \eqref{sch:fast_Collocation}} & \multicolumn{2}{c}{the modified scheme \eqref{sch:fast_collocation_md}} \\ 
		\cmidrule(r){3-4}     \cmidrule(r){5-6}    	
		  &  & $\|e_N\|_{\infty}$ & Cov. & $\|e_N\|_{\infty}$ & Cov.  \\
        \midrule
		  & $2^6$    & 1.1918e-01 & --- & 2.1581e-02 & --- \\
$1$     & $2^7$    & 1.0343e-01 & 0.2045 & 1.8620e-02 & 0.2129 \\
		& $2^8$    & 8.9907e-02 & 0.2022 & 1.6139e-02 & 0.2063 \\
        & $2^9$    & 7.8210e-02 & 0.2011 & 1.4019e-02 & 0.2031 \\[0.03in]
        & &\multicolumn{2}{c}{$\kappa \sigma =0.20$ } & \multicolumn{2}{c}{$\kappa \sigma =0.20$ } \\
        \midrule
		  & $2^6$    & 8.8054e-02 & --- & 7.7713e-02 & --- \\
$2/(1+\sigma)$ & $2^7$ & 6.9829e-02 & 0.3346 & 6.1674e-02 & 0.3335 \\
		  & $2^8$    & 5.5406e-02 & 0.3338 & 4.8949e-02 & 0.3334 \\
		  & $2^9$    & 4.3971e-02 & 0.3335 & 3.8850e-02 & 0.3334 \\[0.03in]
		& &	\multicolumn{2}{c}{$ \kappa \sigma \approx 0.33$ }   
		   & \multicolumn{2}{c}{$  \kappa \sigma \approx 0.33$ } \\
           \midrule
		  & $2^6$    & 3.0355e-02 & --- & 2.0859e-01 & --- \\
$(2-\alpha)/2\sigma$ & $2^7$ & 1.7439e-02 & 0.7996 & 1.1981e-01 & 0.7999 \\
	    & $2^8$    & 1.0017e-02 & 0.7999 & 6.8812e-02 & 0.8000 \\
		& $2^9$    & 7.5735e-03 & 0.8000 & 3.9522e-02 & 0.8000 \\[0.03in]
        & &\multicolumn{2}{l}{$\min\{ 2-\alpha, \kappa \sigma \} =\kappa \sigma =0.80$ } & \multicolumn{2}{l}{$\min\{ 2-\alpha, \kappa \sigma \} =\kappa \sigma=0.80$ } \\
        \midrule
		  & $2^6$    & 1.0649e-02 & --- & 7.7674e-01 & --- \\
	$(2-\alpha)/\sigma$ & $2^7$ & 3.6670e-03 & 1.5380 & 2.5633e-01 & 1.5994 \\
		  & $2^8$    & 1.2299e-03 & 1.5761 & 8.4590e-02 & 1.5994 \\
		& $2^9$    & 4.0736e-04 & 1.5942 & 2.7914e-02 & 1.5995 \\[0.03in]
		& &	\multicolumn{2}{l}{$ \min\{ 2-\alpha, \kappa \sigma \} =2-\alpha =1.60$ }   
		   & \multicolumn{2}{l}{$ \min\{ 2-\alpha, \kappa \sigma \} =2-\alpha =1.60$ } \\
		\bottomrule
	\end{tabular}
\end{table}	

Next, we evaluate the efficiency of the proposed fast collocation scheme. In Tables \ref{tab:eff-1}--\ref{tab:eff-2}, we present the errors and CPU times for the direct collocation scheme \eqref{sch:direct_Collocation} solved by the Gaussian elimination solver (denoted as D-GE), the direct collocation scheme solved by the BiCGSTAB iterative solver (denoted as D-BiCGSTAB), the fast collocation scheme \eqref{sch:fast_Collocation} solved by the fast version BiCGSTAB iterative solver (denoted as F-BiCGSTAB), and the preconditioned fast collocation scheme solved by the fast version BiCGSTAB iterative solver (denoted as PF-BiCGSTAB). Meanwhile, the average number of iterations (Iter.) is also listed for the iterative solver. We can reach the following observations: (i) All these schemes with different solvers generate almost identical numerical solutions; (ii) The PF-BiCGSTAB algorithm takes significantly less CPU time than the other three algorithms, especially for large-scale modeling. For example, in the case of $\alpha=0.9$, $\kappa=1$ and $N=2^{13}$, the CPU time consumed by the D-GE algorithm is more than 20 minutes! In contrast, the PF-BiCGSTAB algorithm only costs no more than 2 seconds; (iii) The use of nonuniform grids does not compromise the computational accuracy of the PF-BiCGSTAB algorithm, and it nevertheless remains the most efficient among all the algorithms compared. (iv) It is clear from Table \ref{tab:eff-3} that round-off errors have begun to noticeably affect the numerical results of the D-BiCGSTAB algorithm for $N=2^{13}$. In fact, under this scenario the corresponding numerical solution diverges. In summary, all these results show the superiority of the PF-BiCGSTAB algorithm with the use of the developed fast matrix-vector multiplication and preconditioner.
\begin{table}[!thbp]\small
\vspace{-12pt}
	\centering  
	\caption{Performance of different algorithms for $\alpha=0.9,~ \kappa=1$}\label{tab:eff-1}
	\setlength{\tabcolsep}{3mm}
	\begin{tabular}{ c c c c c c c}
		\toprule
		\multirow{2}{*}{$N$}&  \multicolumn{2}{c}{D-GE} & \multicolumn{3}{c}{D-BiCGSTAB} \\ 
		\cmidrule(r){2-4}     \cmidrule(r){5-7}    
		& $\|e_N\|_{\infty}$ & CPU & & $\|e_N\|_{\infty}$ & CPU  & Iter.   \\	
		\midrule
		$2^{10}$ & 5.0874e-02 & 0.70s  &  & 5.0874e-02 & 0.56s  & 159   \\
		$2^{11}$ & 3.7159e-02 & 10s &  & 3.7159e-02 & 7.11s  & 206   \\
		$2^{12}$ & 2.7169e-02 & 145s & & 2.7169e-02 & 43s & 285  \\ 
		$2^{13}$ & 1.9876e-02 & 1396s& & 1.9876e-02 & 294s   & 404    \\
		\midrule
		\multirow{2}{*}{$N$} & \multicolumn{2}{c}{F-BiCGSTAB} & \multicolumn{3}{c}{PF-BiCGSTAB} \\ 
		\cmidrule(r){2-4}     \cmidrule(r){5-7}    
		& $\|e_N\|_{\infty}$ & CPU & Iter. & $\|e_N\|_{\infty}$ & CPU  & Iter.   \\	
		\midrule
		$2^{10}$ & 5.0871e-02 & 0.20s  & 155 & 5.0871e-02 & 0.05s  & 36   \\
		$2^{11}$ & 3.7157e-02 & 0.55s  & 208 & 3.7157e-02 & 0.13s  & 45   \\
		$2^{12}$ & 2.7168e-02 & 1.53s  & 281 & 2.7168e-02 & 0.33s  & 58  \\ 
		$2^{13}$ & 1.9876e-02 & 9.03s  & 444 & 1.9876e-02 & 1.42s  & 78    \\
		\bottomrule
	\end{tabular}
\end{table}	
\begin{table}[!thbp]\small
\vspace{-12pt}
		\centering  
        \caption{Performance of different algorithms for $\alpha=0.9, ~\kappa=\frac{2-\alpha}{2\sigma}$}\label{tab:eff-2}
        \setlength{\tabcolsep}{3mm}
		\begin{tabular}{ c c c c c c c}
			\toprule
			\multirow{2}{*}{$N$}&  \multicolumn{2}{c}{D-GE} & \multicolumn{3}{c}{D-BiCGSTAB} \\ 
			\cmidrule(r){2-4}     \cmidrule(r){5-7}    
			& $\|e_N\|_{\infty}$ & CPU & & $\|e_N\|_{\infty}$ & CPU  & Iter.   \\	
			\midrule
			$2^{10}$ & 3.3811e-02 & 0.78s  &  & 3.3811e-02 & 0.95s  & 229   \\
			$2^{11}$ & 2.3051e-02 & 11s &  & 2.3051e-02 & 12s  & 336   \\
			$2^{12}$ & 1.5730e-02 & 139s & & 1.5730e-02 & 73s & 479  \\ 
			$2^{13}$ & 1.0739e-02 & 1409s& & 1.0739e-02 & 504s   & 712    \\
			\midrule
			\multirow{2}{*}{$N$}  & \multicolumn{2}{c}{F-BiCGSTAB} & \multicolumn{3}{c}{PF-BiCGSTAB} \\ 
			\cmidrule(r){2-4}     \cmidrule(r){5-7}    
			& $\|e_N\|_{\infty}$ & CPU & Iter. & $\|e_N\|_{\infty}$ & CPU  & Iter.   \\	
			\midrule
			$2^{10}$  & 3.3810e-02 & 0.28s  & 206 & 3.3810e-02 & 0.06s  & 39   \\
			$2^{11}$  & 2.3051e-02 & 0.83s  & 297 & 2.3051e-02 & 0.14s  & 55   \\
			$2^{12}$  & 1.5729e-02 & 2.69s  & 468 & 1.5729e-02 & 0.38s  & 63  \\ 
			$2^{13}$  & 1.0739e-02 & 14s  & 672 & 1.0739e-02 & 1.53s  & 76    \\
			\bottomrule
		\end{tabular}
\end{table}
\begin{table}[htbp!]\small
\vspace{-12pt}
		\centering  
        \caption{Performance of different algorithms for $\alpha=0.9, ~\kappa=\frac{(2-\alpha)}{\sigma}$}\label{tab:eff-3}
        \setlength{\tabcolsep}{3mm}
		\begin{tabular}{ c c c c c c c}
			\toprule
			\multirow{2}{*}{$N$}&  \multicolumn{2}{c}{D-GE} & \multicolumn{3}{c}{D-BiCGSTAB} \\ 
			\cmidrule(r){2-4}     \cmidrule(r){5-7}    
			& $\|e_N\|_{\infty}$ & CPU & & $\|e_N\|_{\infty}$ & CPU  & Iter.   \\	
			\midrule
			$2^{10}$ & 9.1320e-03 & 0.77s  &  & 9.1320e-03 & 18s  & 5011   \\
			$2^{11}$ & 4.2708e-03 & 10s &  & 4.2708e-03 & 587s  & 17359   \\
			$2^{12}$ & 1.9946e-03 & 139s & & 1.9946e-03 & 6096s & 40372  \\ 
			$2^{13}$ & 9.3096e-04 & 1383s& &  N/A&    &     \\
			\midrule
			\multirow{2}{*}{$N$}  & \multicolumn{2}{c}{F-BiCGSTAB} & \multicolumn{3}{c}{PF-BiCGSTAB} \\ 
			\cmidrule(r){2-4}     \cmidrule(r){5-7}    
			& $\|e_N\|_{\infty}$ & CPU & Iter. & $\|e_N\|_{\infty}$ & CPU  & Iter.   \\	
			\midrule
			$2^{10}$  & 9.1320e-03 & 5.80s  & 2663 & 9.1320e-03 & 0.05s  & 25   \\
			$2^{11}$  & 4.2708e-03 & 23s  & 5674 & 4.2708e-03 & 0.16s  & 37   \\
			$2^{12}$  & 1.9946e-03 & 140s  & 17082 & 1.9946e-03 & 0.44s  & 50  \\ 
			$2^{13}$  & 9.3096e-04 & 2001s  & 77835 & 9.3096e-04 & 1.86s  & 68    \\
			\bottomrule
		\end{tabular}
\end{table}

\begin{figure}[!thbp]\small
	\centering
	\subfigure[$\kappa=1$]{\includegraphics[width=0.45\textwidth]{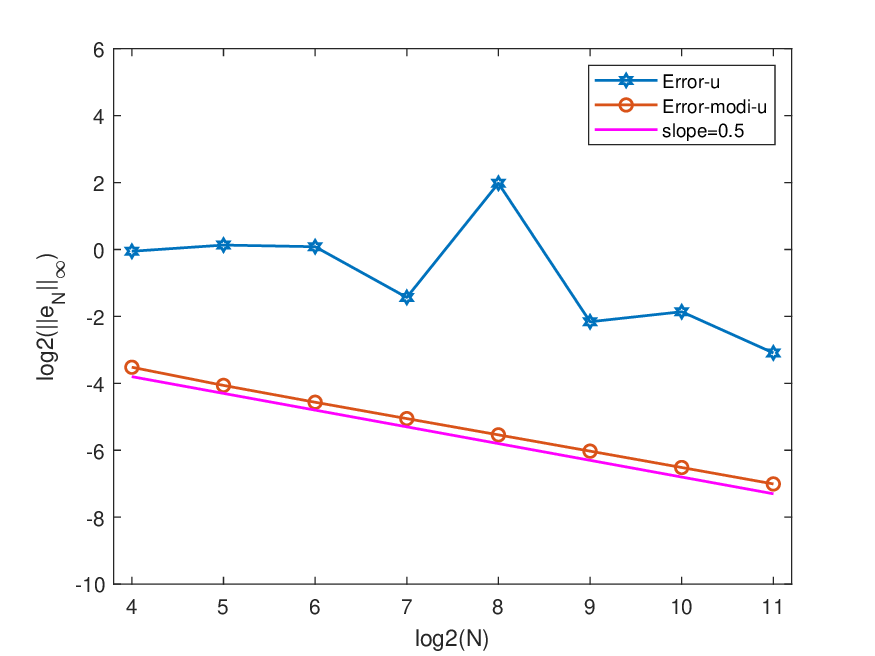}}
	\subfigure[$\kappa=\frac{2(2-\alpha)}{\alpha}$]{\includegraphics[width=0.45\textwidth]{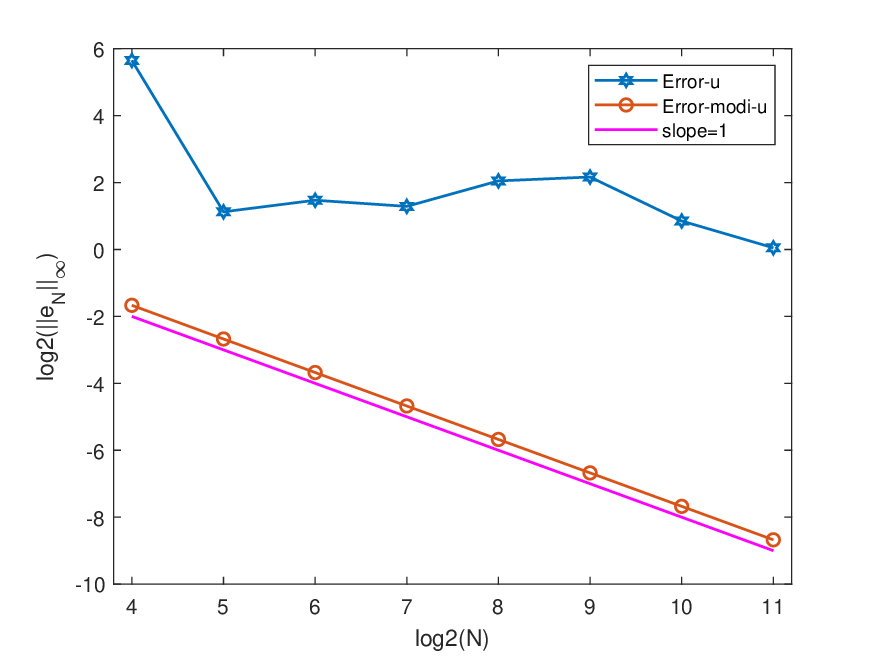}}
    \subfigure[$\kappa=1$]{\includegraphics[width=0.45\textwidth]{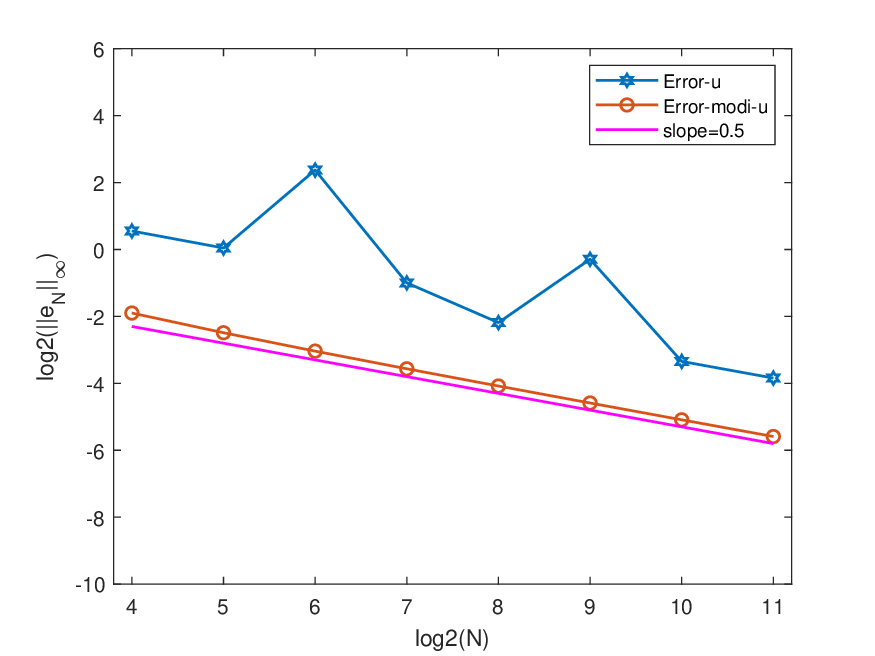}}
	\subfigure[$\kappa=\frac{2(2-\alpha)}{\alpha}$]{\includegraphics[width=0.45\textwidth]{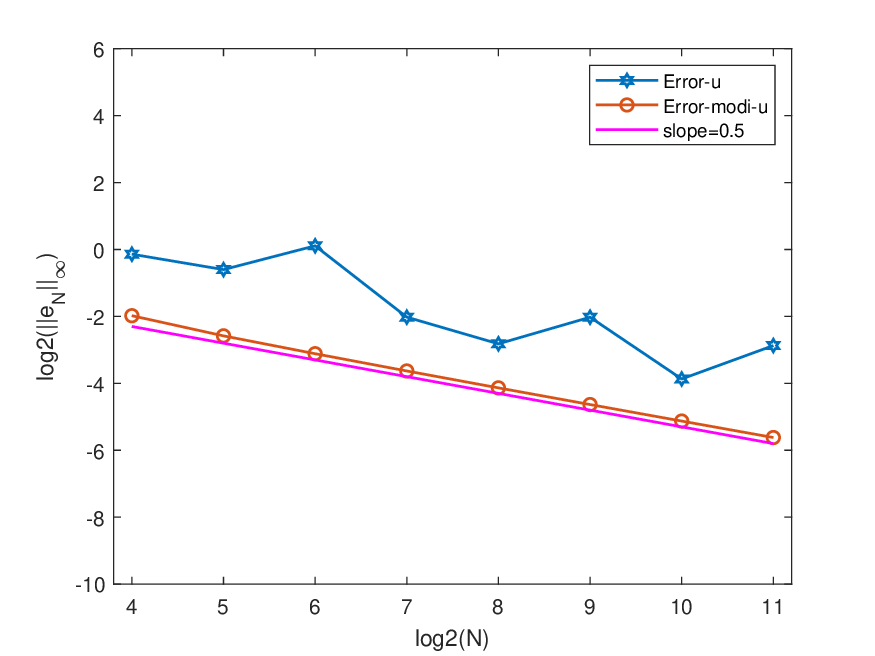}}
    \subfigure[$\kappa=1$]{\includegraphics[width=0.45\textwidth]{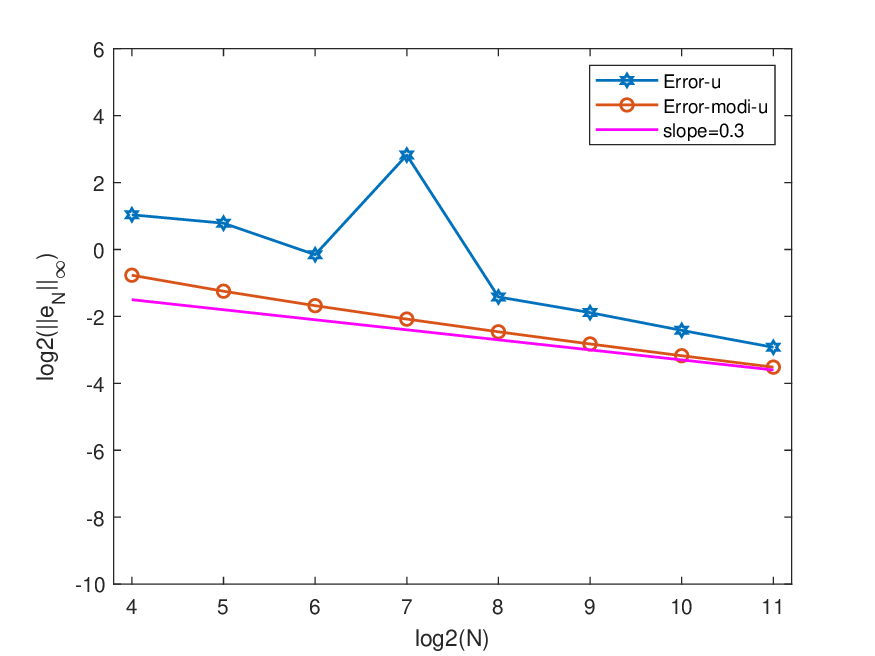}}
	\subfigure[$\kappa=\frac{2(2-\alpha)}{\alpha}$]{\includegraphics[width=0.45\textwidth]{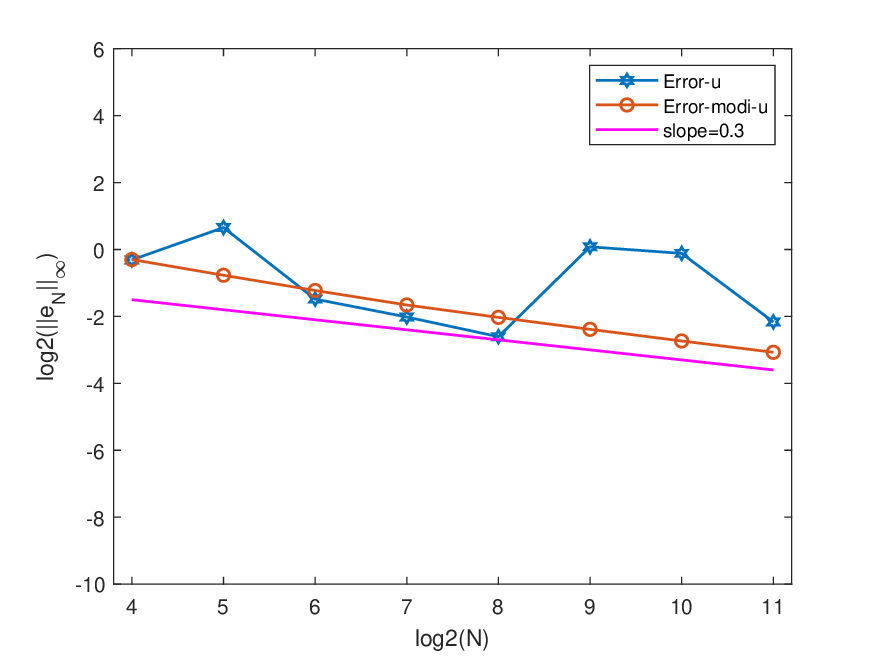}}
    \setlength{\abovecaptionskip}{0.0cm}
	\setlength{\belowcaptionskip}{0.0cm}
	\caption{Comparisons of the two fast collocation schemes. From top to bottom: $\alpha=1.0, 1.5, 1.7$ }
	\label{fig:alp_1}
\end{figure}
Although no rigorous theoretical conclusion is provided for the case $\alpha\in[1,2)$, we present numerical results to show the accuracy of the developed fast collocation schemes \eqref{sch:fast_Collocation} and \eqref{sch:fast_collocation_md}. Figure \ref{fig:alp_1} displays the convergence rates for different values of $\alpha \in[1,2)$. These results confirm that the modified scheme also achieves the expected convergence rate of $\mO(N^{-\min\{2-\alpha,\kappa\sigma\}})$. In contrast, the original fast collocation scheme yields less satisfactory results and, in some cases, even fails to produce reliable numerical approximations. This discrepancy is primarily attributed to the improved local approximation offered by the modified scheme. Therefore, the modified fast scheme \eqref{sch:fast_collocation_md} is recommended as a robust choice for model problem with $\alpha \in [1,2)$, regardless of whether uniform or nonuniform grids are employed. Furthermore, as seen when $\alpha\geq 4/3$, the advantage of using nonuniform grids diminishes, since in this case the achievable convergence rate is the same as that on uniform grids, i.e., $\min\{2-\alpha,\kappa\sigma\} = 2-\alpha$.

\section{Conclusions}
A fast collocation scheme was proposed for the numerical simulation of integral fractional Laplacian problem. The developed method is well-suited for general nonuniform grids. By carefully analyzing the mathematical structure of the coefficient matrix, we proved that the proposed scheme is uniquely solvable on general nonuniform grids for $\alpha\in(0,1)$. In addition, a modified scheme was introduced by improving the local-part approximation, and its unique solvability on uniform grids was analyzed. Furthermore, efficient implementations of both schemes based on any Krylov subspace iterative solver and fast matrix-vector multiplication without matrix assembling were carefully discussed. To further improve computational efficiency, a simple banded preconditioner was introduced to accelerate the iterative procedure. 

We also established a maximum nodal error bound of order $\mO(N^{-\min\{2-\alpha,\kappa\sigma\}})$ for the fast collocation scheme \eqref{sch:fast_Collocation} and the modified scheme \eqref{sch:fast_collocation_md} on symmetric graded grids. The analysis indicated that a mesh grading parameter $\kappa \geq (2-\alpha)/\sigma$ is required to achieve the optimal convergence order $\mO(N^{-(2-\alpha)})$. However, the current error analysis is only restricted to the case $\alpha\in(0,1)$, and a rigorous theoretical extension to $\alpha\in[1,2)$ and to general nonuniform grids will be pursued in future work.
Numerical experiments were provided to validate the theoretical results. We can observe that
the proposed fast collocation schemes achieve accuracy comparable to the direct collocation scheme, while significantly reducing both computational complexity and memory consumption. 
The findings also reveal that the modified scheme outperforms the original one in most scenarios mainly due to the better local approximations. 

\section*{CRediT authorship contribution statement}
\textbf{Meijie Kong}: Methodology, Formal analysis, Software, Writing- Original draft.
\textbf{Hongfei Fu}: Conceptualization, Supervision, Methodology, Writing- Reviewing and Editing, Funding acquisition.

\section*{Declaration of competing interest}
The authors declare that they have no competing interests.

\section*{Data availability } Data will be made available on request.

\section*{Acknowledgments}
This work was supported in part by the National Natural Science Foundation of China (Nos. 11971482, 12131014), by the Shandong Provincial Natural Science Foundation (No. ZR2024MA023), by the Fundamental Research Funds for the Central Universities (No. 202264006) and by the OUC Scientific Research Program for Young Talented Professionals.

\bibliographystyle{plain}
\bibliography{myref}
\end{document}